\documentclass[12pt]{amsart}
\usepackage{a4wide,amssymb,stmaryrd,amscd}
\usepackage[usenames]{color}
\usepackage[all]{xy}
\newcommand{\Z}{\mathbb{Z}}

\newcommand{\C}{\mathbb{C}}

\newcommand{\even}{\operatorname{even}}
\newcommand{\odd}{\operatorname{odd}}
\newcommand{\disj}{orthogonal}
\newcommand{\rk}{\operatorname{rk}}
\newcommand{\basis}{\mathcal{B}}
\newcommand{\der}{\partial}
\newcommand{\join}{\vee}
\newcommand{\arr}{\mathcal{H}}
\newcommand{\disjoin}{\ovee}
\newcommand{\irr}{\operatorname{ind}}
\newcommand{\dec}{\operatorname{dec}}
\newcommand{\data}{\mathcal{D}}
\newcommand{\ideal}{\mathcal{I}}
\newcommand{\card}[1]{\lvert{#1}\rvert}
\newcommand{\id}{\operatorname{id}}
\newcommand{\atoms}{\mathbb{A}}
\newcommand{\dsum}{\oplus}
\newcommand{\Orlik}{\mathcal{O}}
\newcommand{\alg}{\tilde{\mathcal{A}}}

\renewcommand{\Im}{\operatorname{Im}}
\newcommand{\meet}{\wedge}
\newcommand{\proj}{\operatorname{proj}}
\newcommand{\St}{\operatorname{St}}
\newcommand{\Sym}{\operatorname{Sym}}
\newcommand{\Ker}{\operatorname{Ker}}

\newcommand{\admis}{$L$-contractible}
\newcommand{\admiss}{compatible}
\newcommand{\reg}{\operatorname{reg}}
\newcommand{\im}{\operatorname{im}}
\newcommand{\codim}{\operatorname{codim}}
\newcommand{\sign}{\operatorname{sign}}
\newcommand{\Rperp}{\mathfrak{M}}

\newcommand{\zono}{\mathcal{Z}}

\newcommand{\facets}{\mathfrak{E}}

\newcommand{\edges}{\mathcal{E}}

\newcommand{\bnd}{\delta}
\newcommand{\module}{\mathfrak{M}}
\newcommand{\family}{\mathcal{A}}
\newcommand{\rest}{|}
\newcommand{\sprod}[2]{\left\langle#1,#2\right\rangle}
\newcommand{\abs}[1]{\lvert#1\rvert}
\newcommand{\indexset}{\mathbb{I}}
\newcommand{\gen}[2]{\left(#1\right)_{\Sym #2}}
\newcommand{\extalg}{\mathcal{E}}       
\newcommand{\jdeal}{\mathcal{J}}
\newcommand{\OS}{\mathcal{A}}

\newtheorem{theorem}{Theorem}[section]
\newtheorem{lemma}[theorem]{Lemma}
\newtheorem{proposition}[theorem]{Proposition}
\newtheorem{remark}[theorem]{Remark}

\newtheorem{definition}[theorem]{Definition}
\newtheorem{corollary}[theorem]{Corollary}
\newtheorem{example}[theorem]{Example}

\begin{document}

\title[Relation spaces]{Relation spaces of hyperplane arrangements and modules defined by graphs of fiber zonotopes}
\author{Tobias Finis}
\address{Freie Universit\"at Berlin, Institut f\"ur Mathematik, Arnimallee 3, D-14195 Berlin,
Germany}
\thanks{Authors partially sponsored by grant \# 964-107.6/2007 from the German-Israeli Foundation for Scientific Research and Development.
First named author supported by DFG Heisenberg grant \# FI 1795/1-1.}
\email{finis@math.fu-berlin.de}
\author{Erez Lapid}
\address{Einstein Institute of Mathematics, The Hebrew University of Jerusalem, Jerusalem, 91904, Israel, and
Department of Mathematics, Weizmann Institute of Science, Rehovot, 76100, Israel}
\email{erez.m.lapid@gmail.com}
\date{\today}
\keywords{Geometric lattices, hyperplane arrangements, Coxeter groups, fiber polytopes}
\subjclass[2010]{Primary 05E40; Secondary 20F55, 52C35, 16E45, 06C10}

\begin{abstract}
We study the exactness of certain combinatorially defined complexes
which generalize the Orlik-Solomon algebra of a geometric lattice. The main results pertain to complex reflection
arrangements and their restrictions. In particular, we consider the corresponding relation complexes and give
a simple proof of the $n$-formality of these hyperplane arrangements.
As an application, we are able to bound the Castelnouvo-Mumford regularity of certain modules
over polynomial rings associated to Coxeter arrangements (real reflection arrangements) and their restrictions. The modules in question
are defined using the relation complex of the Coxeter arrangement and fiber polytopes of the dual Coxeter zonotope.
They generalize the algebra of piecewise polynomial functions on the original arrangement.
\end{abstract}

\maketitle

\setcounter{tocdepth}{1}
\tableofcontents

\section{Introduction}

This paper deals with two related topics.
The first is the study of the relation complexes of hyperplane arrangements and of several related constructions.
After a summary of some basic combinatorial facts in \S \ref{preliminaries},
we consider in \S \ref{variant} complexes graded by geometric lattices, and specifically an inductive construction of such
complexes starting from a collection of injections $U_a \hookrightarrow U_0$ indexed by the atoms $a$ of the lattice (cf.~Definition \ref{defminimalcomplex}).
We call the complexes obtained from this construction \emph{minimal complexes}, and
our main interest is in criteria for their exactness.
Apart from the well-known Orlik-Solomon algebra (\cite{MR558866}, \cite[Ch.~3]{MR1217488}, \cite{MR1859708}),
an important special case is the relation complex of a hyperplane arrangement,
which describes
the linear relations between its defining linear functionals.
In \S \ref{reflarrangements}, we give a simple proof of the
$n$-formality of restrictions of complex reflection arrangements, i.e.~we
show the exactness of their relation complexes (Theorem \ref{roots}).
In fact, we prove a somewhat stronger result, and provide an explicit contracting homotopy that is compatible with the grading by the
intersection lattice (in our terminology, we show that the relation complex is \emph{\admis}, cf.~Definition \ref{defadmissible}).
While $n$-formality was known for reflection arrangements and restrictions of
Coxeter arrangements as a consequence of their freeness \cite{MR1280576},
our construction of an $L$-homotopy seems to be a new result.
Also, our approach is direct and elementary and uses only the basic properties of complex reflection groups.
We then consider in \S \ref{algstr} a general multilinear algebra construction, which we call the generalized Orlik-Solomon algebra
of a complex graded by a geometric lattice. At the level of complexes, this construction is shown to preserve the property of being \admis.
This allows us to establish the exactness of certain minimal complexes constructed
from relation spaces, if the relation complex itself is \admis\ (Corollaries \ref{CorDefideals} and \ref{CorProj}).

The second topic is given by certain algebras over polynomial rings defined by the labeled graphs of fiber zonotopes.
Here, the main object is an analog of an exact sequence of Bernstein-Lunts
\cite[\S\S15.7-8]{MR1299527} and Brion \cite[p.~12]{MR1431267} for
the algebra of piecewise polynomial functions on a complete simplicial fan, which relates
this algebra to the polynomial functions on all cones.
In \S \ref{zonotope}, we consider fans obtained from hyperplane arrangements in a real vector space $U^*$. In the dual picture, these correspond to
zonotopes (Minkowski sums of line segments) in the dual space $U$. We derive a criterion for the
existence of an exact sequence of Bernstein-Lunts type for modules
defined in a purely combinatorial way by the labeled graph of a zonotope (Proposition \ref{main}). Using the results of the first part,
we then apply this criterion to the projected
arrangements $\arr_P$ governing the combinatorics of intersections of an arrangement $\arr$
of rank $n$ with the parallel translates
of a fixed linear subspace $P^\perp \subset U^*$ of dimension $k$ (which is assumed to satisfy a
general position condition, cf.~Definition \ref{generalposition}).
Dually, we pass from a zonotope $\zono$ dual to $\arr$ to the fiber polytope $\zono_P$ in the sense of \cite{MR1166643} of its projection
in direction $P \subset U$. We use the $1$-skeleton of $\zono_P$ and the relation complex of $\arr$ to
define a certain algebra $\Rperp$ over the symmetric algebra $R$ of the rank $k$ space of the relation complex,
which we call the \emph{$k$-th order relation algebra} of the arrangement $\arr$ (with respect to $P$).
Our main result (Theorem \ref{maintheorem}) is the existence of an exact sequence of Bernstein-Lunts type for $\Rperp$, provided the relation complex of $\arr$ is
\admis\ (in particular if $\arr$ is a restriction of a Coxeter arrangement).
As a consequence, the relation algebra $\Rperp$ is generated as an $R$-module
by its homogeneous elements of degree at most $n-k$.
In the case $k=0$ we recover the algebra
of piecewise polynomial functions on the arrangement. (Note here that restrictions of Coxeter arrangements are simplicial.)

We note that the relation algebra $\Rperp$ has in general a more complicated structure than the algebra of piecewise polynomial functions on a
complete simplicial fan. For example, it is in general not free as a module
over the polynomial ring $R$ (cf.~\S \ref{SubsectionModules}, Remark \ref{maintheoremremark}).
In recent years, combinatorially defined algebras of the type we are considering
have been studied by Guillemin and Zara \cite{MR1701922, MR1823050, MR1959894} following the
work of Goresky-Kottwitz-MacPherson on equivariant cohomology \cite{MR1489894},
but their focus is different, since they are interested in cases where the
algebra in question is actually free over the underlying polynomial ring.


The original motivation for this paper was the study of the push-forward of piecewise power series on polyhedral fans
(a canonical $n$-th order derivative at the origin) in the special case of
so-called compatible families. The problem arose in the study of the spectral side of Arthur's trace formula.
For this application it is necessary to consider families
with values in (in general) non-commutative finite-dimensional algebras of
characteristic zero.
The main result of our work was an alternative formula for the
canonical push-forward of compatible families on restrictions of Coxeter arrangements.
Eventually, we found a simpler proof of this formula that does not use the results of this paper
and moreover applies to all polyhedral fans.
We refer the interested reader to \cite{MR2811598}
for the precise algebraic formula in the new setup and to \cite{MR2811597} for the application to the trace formula.
Nevertheless, we hope that the combinatorial material presented in this paper is of independent interest.
A sketch of the original approach and its connection to Theorem \ref{maintheorem} is contained in \S \ref{concluding}.

In \S \ref{higherbruhat} we explicate our construction in the special case of the root systems $A_n$, where we can draw a
connection to the theory of discriminantal hyperplane arrangements and higher Bruhat orders.
The appendix contains an explicit description
of relation complexes and $L$-homotopies for the restrictions of non-exceptional Coxeter arrangements.

We thank Joseph Bernstein, Tom Braden, Stefan Felsner, Gil Kalai, Allen Knutson, Eric Opdam, Ehud de Shalit, Sergey Yuzvinsky
and G\"unter Ziegler for useful discussions
and interest in the subject matter of this paper.
We also thank the Max Planck Institute for Mathematics and the Hausdorff Research Institute for Mathematics, Bonn, where
a part of this paper was worked out.
Finally, we are indebted to the referee for a careful reading of the manuscript and for useful suggestions.

A previous version of this paper forms a part of the first named author's 2009 Habilitation at Heinrich-Heine-Universit\"{a}t D\"{u}sseldorf.


\section{Geometric lattices and hyperplane arrangements} \label{preliminaries}
In this section we collect some mostly standard facts about geometric lattices and hyperplane arrangements
and set up some notation. We refer the reader to \cite[Ch.~II]{MR1434477} and \cite{MR1217488} for more details.

\subsection{Basic definitions}
For any poset $(L,<)$
we write $x\prec y$, or equivalently $y\succ x$, if $y$ covers $x$,
i.e.~if $x<y$ and there is no $z \in L$ with $x < z < y$.
As usual, we denote the join and meet in a lattice $L$ by $\join_L$ and $\meet_L$, respectively,
or simply by $\join$ and $\meet$, if $L$ is clear from the context. In the following we always assume that $L$ is finite.
Recall that a (finite) lattice $L$ is called \emph{geometric}, if it is atomic (i.e.~every element is
the join of atoms) and semimodular (i.e.~$a \meet b \prec a$ implies $b \prec a \join b$ for all $a$, $b \in L$).
Semimodularity is equivalent to the existence of a rank function $\rk=\rk_L$ satisfying the inequality
$\rk (a \meet b) + \rk (a \join b) \le \rk (a) + \rk (b)$ \cite[Theorem 2.27]{MR1434477}.
Let now $L$ be a geometric lattice of rank $\rk L = n$ with minimal element $0$.
Set
\[
L_i=\{x\in L:\rk(x)=i\},\ \ \ i=0,\dots,n,
\]
and denote by $\atoms(L)=L_1$ the set of atoms of $L$.
For any $x\in L$ the lower interval $L_{\le x}=\{y\in L:y\le x\}$
is again a geometric lattice (with the restricted rank function), whose set of atoms is $\atoms_{\le x}=\{a\in\atoms:a\le x\}$.
We write $L_{\le x,i}=(L_{\le x})_i=L_i\cap L_{\le x}$.
Similarly, the upper interval $L_{\ge x}=\{y\in L:y\ge x\}$ is a geometric lattice with minimal element $x$
and $\rk_{L_{\ge x}}(y)=\rk_L(y)-\rk_L(x)$ for any $y\in L_{\ge x}$.

The basic example of a geometric lattice, which provides most of the intuition,
is the intersection lattice of a hyperplane arrangement. More precisely,
let $K$ be a field, $U$ a finite-dimensional vector space over $K$
and $\arr$ a finite collection of hyperplanes in the dual vector space $U^*$.
The set of all possible intersections of elements of $\arr$ forms a geometric lattice
under inverse inclusion. The minimal element of the lattice is $U^*$ itself (the empty intersection),
the atoms are the elements of $\arr$ and the maximal element is the intersection of all hyperplanes
in $\arr$. The join is the intersection and the rank function is the codimension in $U^*$.
Dually, the map
\[
H\mapsto H^\perp=\{u\in U:\sprod\lambda u=0\text{ for all }\lambda\in H\}
\]
defines a bijection between $\arr$ and a set of lines $\tilde\arr$ in $U$.
Sometimes, it is convenient to fix a choice of a non-zero vector $u_H\in H^\perp$ for each $H\in\arr$.
The lattice $L$ is isomorphic to the lattice $\tilde L$ consisting of the
linear spans of subsets (the span lattice) of $\tilde\arr$.
In $\tilde L$, the minimal element is the zero subspace, the join is the sum and the rank function is the dimension.
For any $x\in L$, the lower interval $L_{\le x}$ is the intersection lattice of the hyperplane arrangement
$\arr_{\le x}$ in the space $U^*$ consisting of the hyperplanes in $\arr$ containing $x$.
Similarly, $L_{\ge x}$ is the intersection lattice of the \emph{restricted hyperplane arrangement} $\arr_{\ge x}$
in the space $x$ given by
$\{H\cap x:H\in\arr,H\not\supset x\}$. Dually, $\widetilde{L_{\le x}}$ is the span lattice in $x^\perp$ of all lines $H^\perp$,
$H\supset x$, while $\widetilde{L_{\ge x}}$ is the span lattice in $U/x^\perp$ of the projections modulo $x^\perp$
of $H^\perp$, $H\not\supset x$.

\subsection{Decomposability and dependency}
Let $L$ be a geometric lattice and $x\in L\setminus\{0\}$. Recall that for $y,z\in L$ the following conditions are equivalent
(cf.~\cite[Theorem 2.45]{MR1434477}):
\begin{enumerate}
\item $y\join z=x$, $y\meet z=0$ and $y$ is a distributive element of $L_{\le x}$, i.e. the relations
$y \meet (a \join b) = (y \meet a) \join (y \meet b)$ and
$a \meet (y \join b) = (a \meet y) \join (a \meet b)$, as well as the corresponding dual relations obtained by
interchanging $\meet$ and $\join$, hold for all $a, b \in L_{\le x}$ (cf. [ibid., p. 58]).
\item $\atoms_{\le x}=\atoms_{\le y}\sqcup\atoms_{\le z}$.
\item The join gives an isomorphism of geometric lattices between $L_{\le y}\times L_{\le z}$
and $L_{\le x}$.
\end{enumerate}
%
%
In this case we write $x=y\disjoin z$.

\begin{definition} We say that $x$ is \emph{decomposable}
if $x=y\disjoin z$ for some $y,z\in L\setminus\{0\}$;
otherwise $x$ is called \emph{indecomposable}.
We say that $L$ is indecomposable if its maximal element is indecomposable.
The sets of decomposable and indecomposable elements of $L\setminus\{0\}$
are denoted by $L_{\dec}$ and $L_{\irr}$, respectively.
\end{definition}

Note that $L_{\irr}\supset\atoms(L)$.
Also, if $0\ne y\le x$, then $y$ is indecomposable as an element of $L$ if and only if
it is indecomposable as an element of $L_{\le x}$; in other words,
$(L_{\le x})_{\irr}=L_{\le x}\cap L_{\irr}$ and $(L_{\le x})_{\dec}=L_{\le x}\cap L_{\dec}$.
We will also denote these sets by $L_{\le x,\irr}$ and $L_{\le x,\dec}$, respectively.
The following lemma is easy to prove (see \cite[Theorem 2.47]{MR1434477} for the first part, which easily implies the
second).

\begin{lemma} \label{irrdecom}
\begin{enumerate}
\item Any $x\in L$ can be written uniquely (up to permutation) as $\disjoin_{i=1}^m x_i$
where the elements $x_i \in L$ are indecomposable. (If $x=0$ then $m=0$.) In particular, $L_{\le x}=L_{\le x_1}\times
\dots\times L_{\le x_m}$.
\item If $x,y\in L_{\irr}$ are not disjoint (that is, if $x\meet y\ne0$) then
$x\join y\in L_{\irr}$.
\end{enumerate}
\end{lemma}

The following notion extends a standard notion for atoms.
\begin{definition} \label{dependentdefinition}
Elements $x_1,\dots,x_m$ of a geometric lattice $L$ are called \emph{dependent} if $\sum_{i=1}^m \rk(x_i)>\rk (\join x_i)$.
Otherwise, i.e.~if $\sum_{i=1}^m \rk(x_i)=\rk (\join x_i)$, then $x_1,\dots,x_m$ are called \emph{independent}.
\end{definition}

We collect some simple properties in the following lemma.
\begin{lemma} \label{deptriv}
Suppose that $x_1,\dots,x_m\in L$ are dependent.
\begin{enumerate}
\item Any sequence containing $x_1,\dots,x_m$ is also dependent.
\item If $y_1,\dots,y_m\in L$ are such that $y_i\ge x_i$ for all $i$,
then $y_1,\dots,y_m$ are also dependent.
\item \label{precdep}
Let $1 \le j \le m$, $y_i = x_i$ for $i \ne j$ and $y_j \prec x_j$.
Then either $y_1,\dots,y_m$ are dependent or $\join y_i = \join x_i$.
\end{enumerate}
\end{lemma}

\subsection{Truncation} \label{truncation}
In \S\S \ref{algstr} and \ref{zonotope} below we will make use of the following type of operation.


\begin{definition} \label{truncdefinition}
Let $L$ and $\Lambda$ be two geometric lattices, and $0 \le k \le n = \rk L$.
We say that a monotone map $l: L_{\ge k} \to \Lambda$ is a \emph{non-degenerate
$k$-th truncation},
if
\begin{enumerate}
\item $\rk_\Lambda (l (x)) = \rk_L (x) - k$ for all $x \in L_{\ge k}$,
\item $l (x \join_L y) = l (x) \join_\Lambda l(y)$ for $x, y \in L_{\ge k}$ with
$\rk_L (x \meet y) \ge k$.
\end{enumerate}
\end{definition}

The basic example is the following.
Let $\arr$ be a hyperplane arrangement in $U^*$ with intersection lattice
$L$, and let $P \subset U$ be a subspace of codimension $k$.
Denote by $\pi_P$ the canonical projection $U^* \to P^*$.
Assume that for all $x \in L_k$ we have
$x \cap P^\perp = \{0\}$. Then
$\codim \pi_P (x) = \rk (x) - k$ for any $x \in L_{\ge k}$,
and in particular the images under
$\pi_P$ of the elements of $L_{k+1}$ (which are not necessarily distinct)
form a new hyperplane arrangement $\arr_P$ in $P^*$.
If $\Lambda$ is the intersection lattice of $\arr_P$, then the
canonical map $l: L_{\ge k} \to \Lambda$, $x \mapsto \pi_P (x)$, is
a non-degenerate $k$-th truncation map. Indeed, we already verified property (1), and if $x$ and $y$ are contained in $z\in L_k$,
then $\pi_P(x\cap y)=\pi_P(x)\cap\pi_P(y)$, since $\pi_P$ is injective on $z$, which gives (2).

For later use, we formulate the following slightly stronger condition on $P$.

\begin{definition} \label{generalposition}
We say that a subspace $P \subset U$ of codimension $k$ is in \emph{general position} (with respect to the hyperplane arrangement $\arr$),
if $x \cap P^\perp = \{0\}$ for all $x \in L_k$, and in addition, the hyperplanes $\pi_P (x)$, $x\in L_{k+1}$ of $P^*$
are distinct.
\end{definition}

Note that under this condition the canonical map $l: L_{\ge k} \to \Lambda$ is a non-degenerate $k$-th truncation map
that maps $L_{k+1}$ bijectively to
the set of atoms of $\Lambda$.

For an arbitrary geometric lattice $L$, the $k$-th Dilworth truncation $D_k (L)$ of $L$ (also called the Dilworth completion) provides another,
purely combinatorial example for Definition \ref{truncdefinition} (see
\cite[pp. 299--303]{MR1434477} for more details). If $L$, $\Lambda$ and $l$ arise from the projection of a hyperplane
arrangement as before, then the constructions are connected as follows: for all $P$ in a non-empty Zariski open subset of the
Grassmannian of all codimension $k$ subspaces of $U$,
the intersection lattice $\Lambda$ of the arrangement $\arr_P$ is isomorphic to $D_k(L)$ \cite{MR843372}.
We will not make use of this fact.



\begin{lemma} \label{trunclemma}
Let $L$ and $\Lambda$ be geometric lattices, $0 \le k \le\rk L$ and $l: L_{\ge k} \to \Lambda$ a non-degenerate
$k$-th truncation map.
Let $\tilde x \in \Lambda$ and let $x_1,\dots,x_m \in L_{\ge k}$ be the maximal elements of $L_{\ge k}$ with the property that
$l(x_i) \le_\Lambda \tilde x$. Then
\[
\{ x \in L_{\ge k}: l(x) \le_\Lambda \tilde x \}=\coprod_{i=1}^m(L_{\le x_i}\cap L_{\ge k}).   
\]
\end{lemma}

\begin{proof}
By the definition of $x_1,\dots,x_m$, we have $l(x)\le_\Lambda\tilde x$ if and only if $x\le_L x_i$ for some $1 \le i \le m$.
On the other hand, by the second property of $l$ and the maximality of the elements $x_i$,
we necessarily have $\rk_L(x_i \meet x_j) < k$ for $i \ne j$.
Therefore, the sets $L_{\le x_i}\cap L_{\ge k}$, $i=1,\dots,m$, are disjoint.
\end{proof}

\section{Minimal complexes graded by geometric lattices} \label{variant}

In this section we define the main object of this paper.
Given a geometric lattice $L$ and some auxiliary data we construct a complex graded by $L$.
These complexes, which encompass several previously known constructions,
are modeled after the relation complexes of \cite{MR1280576} and the minimal complexes of \cite[Ch.~15]{MR1299527}.
In \S \ref{reflarrangements}, we will study the exactness of these minimal complexes by considering so-called $L$-homotopies and apply our
constructions to the relation complexes of complex reflection arrangements and their restrictions.
Throughout, let $L$ be a geometric lattice of rank $n$.

\subsection{Review of the Orlik-Solomon algebra} \label{reviewos}
We begin by reviewing the construction and main properties of the Orlik-Solomon algebra
of a geometric lattice (see \cite[Ch.~3]{MR1217488} and \cite{MR1859708} for more details).\footnote{While the results of \cite[\S 3.1]{MR1217488} quoted in the following are stated
for the intersection lattices of (central) hyperplane arrangements, the arguments are purely combinatorial and work for general geometric lattices.}
Let $M$ be the free Abelian group with basis elements $e_a$ indexed by the atoms $a$ of $L$, and
consider the exterior algebra $\extalg$ of $M$.
It is a graded ring which is free of rank $2^{\card{\atoms(L)}}$ as a $\Z$-module.
More precisely, a $\Z$-basis of its degree $k$ part is provided by the products $e_{a_1}\dots e_{a_k}$, where
$a_1,\dots,a_k$ are pairwise different elements of $\atoms(L)$. (For $k > 1$ these basis elements are only defined up to sign.)
The algebra $\extalg$ admits a unique graded derivation $\der$, i.e. a
linear map $\extalg \to \extalg$ with $\der (uv) = (\der u) v + (-1)^k u (\der v)$
for all $u \in \extalg_k$ and $v \in \extalg$, satisfying
$\der e_a=1$ for all $a\in\atoms$ \cite[Lemma 3.10]{MR1217488}. The derivation $\der$ is homogeneous of degree $-1$.
Let $\ideal$ be the ideal of $\extalg$ generated by the expressions $\der(e_{a_1}\dots e_{a_k})$, where
$a_1,\dots,a_k$ are dependent elements of $\atoms (L)$.
One shows [ibid., Lemma 3.15] that $\ideal=\jdeal+\der\jdeal$, where $\jdeal$ is the submodule
of $\extalg$ generated by the products $e_{a_1}\dots e_{a_k}$ for dependent $a_1,\dots,a_k \in \atoms (L)$.
In particular, $\ideal$ is a homogeneous ideal of $\extalg$ and $\der\ideal \subset \ideal$.
By definition, the Orlik-Solomon algebra $\OS(L)$ is the quotient $\extalg/\ideal$.
By the above, the derivation $\der$ induces a derivation on $\OS(L)$ which we still denote by $\der$ (cf. [ibid., Definition 3.12]). Therefore,
$(\OS(L),\der)$ is a differential graded $\Z$-algebra.

The algebra $\OS(L)$
admits a canonical grading by the lattice $L$. Namely, we can write $\OS(L)=\oplus_{x\in L}\OS(L)_x$,
where $\OS(L)_x$ is the image under the canonical projection $\extalg\rightarrow\OS(L)$
of the submodule of $\extalg$ spanned by the products $e_{a_1}\dots e_{a_k}$
with $\join_ia_i=x$ [ibid., Theorem 3.26].
The $\Z$-modules $\OS (L)_x$ are free [ibid., Corollary 3.44] and
we have $\OS(L)_i=\oplus_{x\in L_i}\OS(L)_x$ for all $i=0,\dots,n$ [ibid., Corollary 3.27].
Moreover, it is clear from the definition that
$\der(\OS(L)_x)\subset\oplus_{y\prec x}\OS(L)_y$ for all $x\in L$.
Note that we have $\OS (L)_0 \simeq \Z$ and $\OS (L)_a \simeq \Z$ for all $a \in \atoms (L)$.

By [ibid., Lemma 3.13], the derivation $\der$ on $\OS (L)$ yields an exact sequence
\begin{equation} \label{eq: exctOS}
0\rightarrow\OS(L)_{\rk L}\rightarrow\dots\rightarrow\OS(L)_1\rightarrow\OS(L)_0\rightarrow0,
\end{equation}
and for any $a\in\atoms(L)$ multiplication by $e_a$ is a contracting homotopy for this
sequence.
Furthermore, we have $\OS(L_{\le x})_y=\OS(L)_y$
for all $y\le x$ and therefore $\OS(L_{\le x})=\oplus_{y\le x}\OS(L)_y$ [ibid., Proposition 3.30]. Applying \eqref{eq: exctOS} to $L_{\le x}$, we infer that
the restriction of $\der$ to $\OS(L)_x$ provides an isomorphism between $\OS(L)_x$ and the kernel
of $\der$ on $\oplus_{y\prec x}\OS(L)_y$.
It is this point of view that we are going to generalize below.

\begin{example} Let $n = \rk L = 2$. Then $\OS (L) = \OS (L)_0 \oplus \OS (L)_1 \oplus \OS (L)_2$. Recall that
$\OS (L)_0 \simeq \Z$ and that a $\Z$-basis of $\OS (L)_1$ is given by the generators $e_a$, $a \in \atoms (L)$.
Any two distinct atoms $a_1 \neq a_2$ of $L$ are independent, while any three atoms are dependent.
This implies that for any $b \in \atoms (L)$ a $\Z$-basis of $\OS (L)_2$ is given by the products $e_a e_b$
for $a \in \atoms (L) \setminus \{ b \}$. In particular, the rank of $\OS (L)_2$ is equal to $\card{\atoms(L)} - 1$.
\end{example}

\subsection{The basic construction} \label{sec: basiconstr}
From now on let $R$ be a (not necessarily commutative) ring with $1$.

\begin{definition} \label{compatible}
A \emph{compatible $L$-grading} on a chain complex
\begin{equation} \label{cc}
V_{n+1}=0\rightarrow V_n\rightarrow\dots \rightarrow V_i\xrightarrow{\der_i}V_{i-1}\rightarrow\dots\rightarrow V_0
\end{equation}
of $R$-modules is given by submodules $V_x$ of $V_{\rk (x)}$, $x\in L$, such that
\[
V_i=\dsum_{\rk(x)=i}V_x,\ \ \ i=0,\dots,n,
\]
and
\[
\der_{\rk(x)} (V_x)\subset V_{\prec x}:=\dsum_{y\prec x}V_y,\ \ \ x\in L \setminus \{ 0 \}.
\]
\end{definition}

Note that the notation $V_0$ is unambiguous and that $V_n=V_{x_{\rm max}}$ where $x_{\rm max}$ is the maximal element of $L$.
By abuse of language, we sometimes refer to $V=\dsum_{x\in L}V_x$ itself as the graded complex.
For any $x\in L$ we write $\der_x:V_x\rightarrow V_{\prec x}$ for the restriction of $\der_{\rk(x)}$ to $V_x$.
Recall that $V$ is called \emph{acyclic} if $\Ker\der_i=\Im\der_{i+1}$, $i=1,\dots,n$.
(We do not require that $\der_1$ is onto.)
If $V=\dsum_{x\in L}V_x$ is a chain complex with a compatible $L$-grading, then for any $x\in L$
the subcomplex $V_{\le x}:=\dsum_{y\le x}V_y$ inherits a compatible grading by $L_{\le x}$.

The following two definitions are central for this paper.

\begin{definition}
An \emph{atomic datum} $\data$ (over $R$) with respect to $L$ consists of an $R$-module $U_0$
together with submodules $U_a$ for each atom $a$ of $L$.
We set $U^L:=\sum_{a\in\atoms(L)}U_a$ and $U^x:=\sum_{a\le x}U_a$ for $x\in L$.
An atomic datum is called \emph{essential} if $U_0=U^L$.
It is called \emph{non-degenerate} if $U_a\ne0$ for all $a\in\atoms$.
\end{definition}

Given $\data$, we associate to any $x\in L$
an atomic datum with respect to $L_{\le x}$ by setting
$\data_{\le x} := (U_0,(U_a)_{a\in\atoms_{\le x}})$.

\begin{definition} \label{defminimalcomplex}
Given an atomic datum $\data=(U_0,(U_a)_{a\in\atoms(L)})$, there exists a unique chain complex
$V=\dsum_{x\in L}V_x$ with compatible $L$-grading that satisfies the following properties:
\begin{enumerate}
\item $V_0=U_0$,
\item for all $a\in\atoms(L)$ we have $V_a=U_a$ and $\der_a$ is the inclusion
$U_a\hookrightarrow U_0$,
\item for all $i>1$ and $x\in L_i$ we have
\begin{equation} \label{defmincmplx}
V_x=\Ker\der_{\prec x},
\end{equation}
where $\der_{\prec x}:=\der_{i-1}\rest_{V_{\prec x}}$,
and $\der_x$ is the inclusion $V_x\hookrightarrow V_{\prec x}$.
\end{enumerate}
We call $V$ the \emph{minimal complex}
of the atomic datum $\data$, and denote it by $\Orlik(\data)=\Orlik(\data,L)$.
\end{definition}

Note that in a minimal complex the maps $\der_x$ are injective for all $x\in L$, and that in particular $\der_n$ is injective.
Also, observe that for any $x\in L$ we have $\Orlik(\data,L)_{\le x}=\Orlik(\data_{\le x},L_{\le x})$.
Therefore, $\Orlik(\data,L)_x$ depends only on $\data_{\le x}$.

We remark that the notion of a minimal complex makes sense in any abelian category. In particular, if we consider
an atomic datum $\data$ of graded vector spaces or of graded modules over a polynomial ring, then the associated minimal complex $\Orlik(\data)$ inherits the
grading of $\data$ and can therefore be regarded as a bigraded complex.

Our motivating problem in the following is to provide sufficient conditions for the acyclicity of $\Orlik(\data)$.
A simple observation is that for $n = \rk L \le 2$ the complex $\Orlik(\data,L)$ is always acyclic.

\begin{remark} \label{remuniq}
Suppose that a complex \eqref{cc} admits a compatible
$L$-grading and that $\der_a$ is injective for each $a\in\atoms(L)$.
Assume in addition that for all $x\in L$ the subcomplex $V_{\le x}$ is acyclic.
Then by the uniqueness of the minimal complex it follows that
$V$ is isomorphic to $\Orlik(V_0,(V_a)_{a \in \atoms(L)})$ as an $L$-graded complex.
\end{remark}

\begin{definition}
We say that a chain complex $V$ with compatible $L$-grading is \emph{supported on indecomposables} if $V_x=0$ for all
$x \in L_{\dec}$.

An atomic datum $\data=(U_0,(U_a)_{a\in\atoms(L)})$ is called \emph{\disj} if
$U_a\cap U_b=\{0\}$ for all atoms $a,b \in \atoms (L)$ for which $a\join b$ is decomposable.
\end{definition}

\begin{lemma} \label{irred}
Let $\data=(U_0,(U_a)_{a\in\atoms(L)})$ be an atomic datum.
Then $\Orlik(\data,L)$ is supported on indecomposables if and only if $\data$ is \disj.
\end{lemma}

\begin{proof}
Note that if $x=a\disjoin b$, $a,b\in\atoms(L)$, then $\Orlik(\data)_x\simeq U_a\cap U_b$.
The ``only if'' part is therefore clear.
To prove sufficiency, suppose that $\data$ is \disj.
We will show by induction on $m=\rk(x)$ that $\Orlik(\data,L)_x=0$ for all $x \in L_{\dec}$.
The base of the induction is the case $m=2$ which is our assumption.
For the induction step, suppose that $m>2$ and $x=y\disjoin z$.
Then $L_{\prec x}=(y\disjoin L_{\prec z})\sqcup(L_{\prec y}\disjoin z)$.
Since $m>2$, $y$ and $z$ cannot both be atoms.
If neither $y$ nor $z$ is an atom then $L_{\prec x}\subset L_{\dec}$ and therefore, by induction hypothesis
$\Orlik(\data)_{\prec x}=0$, and hence $\Orlik(\data)_x=0$.
Otherwise, by interchanging $y$ and $z$ if necessary, we may assume that $z\in\atoms(L)$ but $y\notin\atoms(L)$.
Then $L_{\prec x}= \{y\} \sqcup (L_{\prec y}\disjoin z)\subset\{y\}\cup L_{\dec}$ and
by the induction hypothesis $\Orlik(\data)_{\prec x}=\Orlik(\data)_y$.
Therefore $\Orlik(\data)_x=0$, since $\der_y$ is injective.
\end{proof}

\subsection{Examples} \label{exmplatomic}
\subsubsection{} \label{constantcoefficients}
The most basic example is given by setting $U_a=U_0$ for all $a \in \atoms(L)$.
We call this case the \emph{atomic datum with constant coefficients}.
From the properties of the Orlik-Solomon algebra $\OS (L)$ summarized in \S \ref{reviewos}
it follows that the $L$-graded chain complex $\OS (L) \otimes U_0$ is exact.
Remark \ref{remuniq} implies then that we have
$\Orlik(\data,L)=\OS(L)\otimes U_0$ as $L$-graded chain complexes.
(The algebra structure of $\OS(L)$ will be considered in \S\ref{algstr} below.)

\subsubsection{}
For the remaining three examples we assume that $L$ is the intersection lattice
of a hyperplane arrangement $\arr$ in a vector space $U^*$ over a field $K$.
Recall that the atoms of $L$ correspond to the elements of $\arr$. We call
$\data = (U, (H^\perp)_{H\in\arr})$ the \emph{defining atomic datum} for $\arr$. It is obviously orthogonal.
The resulting minimal complex $V$ is called the \emph{relation complex} of $\arr$.
Note that for $x \in L$ the atomic datum $\data_{\le x}$ is just the defining atomic datum for $\arr_{\le x}$.

The graded pieces of the relation complex are the relation spaces considered in \cite{MR1280576}.
(More precisely, in the
terminology of [ibid., Definition 4.1, 4.3], the graded piece $V_x$ associated to $x \in L$
is $R_{\rk (x)} (\arr_{\le x})$, and the space $R_k (\arr)$ is the kernel of $\der_{k-1}$ on $V_{k-1}$.)
A hyperplane arrangement is called
\emph{$k$-formal} (cf.~[ibid., Definition 4.4]), if the first $k-1$ homology
groups of the complex \eqref{cc} vanish, i.e.~if $\Ker\der_i=\Im\der_{i+1}$ for
$i=1,\dots,k-1$. For example, $2$-formality means that the
linear dependencies among the linear functionals $u_H$, $H\in\arr$,
are generated by those induced by the rank two
elements of the intersection lattice
(or, what amounts to the same, by the linear dependencies among triples of functionals $u_H$);
$3$-formality means that in addition, the relations among the
relations are generated by those induced by rank three elements, and so forth.
$n$-formality (for $n = \rk L$) means that \eqref{cc} is acyclic.

\begin{example} \label{examplerank2}
In the case $n = 2$ we can write the relation complex $V$ as
\[
0 \longrightarrow V_2 \longrightarrow \bigoplus_H K u_H \longrightarrow U,
\]
where $V_2$ is the vector space of all linear dependencies among the $u_H$, i.e.
$V_2 = \{ (\lambda_H)_H \in K^{\card{\arr}}\, : \, \sum_H \lambda_H u_H = 0 \}$. The space $V_2$ has dimension $\card{\arr} - 2$.
\end{example}

\begin{example} \label{examplea3}
The six vectors $(1,0,0)$, $(0,1,0)$, $(0,0,1)$, $(1,1,0)$, $(0,1,1)$ and
$(1,1,1)$ in the space $U = K^3$ define an essential hyperplane arrangement $\arr$ in the space $U^*$, which
we again identify with $K^3$, as usual. The intersection lattice $L$ of the arrangement $\arr$
contains four elements of rank two, namely the lines in $U^*$ spanned by the vectors $(1,0,0)$, $(0,0,1)$, $(1,-1,0)$ and $(0,1,-1)$.
Each of these lines is contained in precisely three hyperplanes of $\arr$ and contributes therefore a one-dimensional piece to the relation complex.
The reader might verify that all linear dependencies among the defining vectors of $\arr$ are linear combinations of the four linear dependencies
among triples.
Therefore the relation complex of $\arr$ is exact. The piece in degree three is one-dimensional, since the space of all linear dependencies
has dimension $6-3=3$, while the rank two piece of the relation complex has dimension four.
(We note that the arrangement $\arr$ is in fact a reflection arrangement of type $A_3$, cf. \S \ref{reflarr} below.)
\end{example}

Other examples (including a generalization of Example \ref{examplea3}) will be discussed in the appendix.

\subsubsection{}
\label{BT} There is a variant of the previous construction.
Let $\Sym U^*$ be the symmetric algebra of $U^*$, canonically isomorphic to the algebra
of polynomials on $U$. We may form the orthogonal atomic datum of graded $K$-vector spaces
$(\Sym U^* \otimes U,(\Sym H\otimes H^\perp)_{H\in\arr})$, where the grading is inherited from $\Sym U^*$.
The minimal complex $W = \Orlik(\Sym U^* \otimes U,(\Sym H\otimes H^\perp)_{H\in\arr})$
is bigraded by the lattice $L$ and the non-negative integers, and
its degree $0$ part with respect to the second grading is just the relation complex $V$ considered above.

The complex $W$ is closely related to the graded $\Sym U$-modules $D_i(\arr)$ defined in \cite[Definition 4.5, 4.6]{MR1280576}.
Namely, the $K$-vector space $W_x$ is the graded dual of the graded $\Sym U$-module $D_{\rk (x)} (\arr_{\le x})$ and for $k \ge 1$
the kernel of $\der_{k-1}$ on $W_{k-1}$ is the graded dual of $D_k (\arr)$.

In particular,
the $\Sym U$-module of $\arr$-derivations
\[
D (\arr) = \{ \theta \in \Sym U \otimes U^* \, : \, \theta (u_H) \in (\Sym U) u_H \text{ for all } H \in \arr \},
\]
where we write $\theta (u)$ for the image of a pair $(\theta,u)$ under the natural bilinear map
$(\Sym U \otimes U^*) \times U \to \Sym U$ (cf. \cite[Proposition 4.8]{MR1217488}),
can be recovered from the map $\der_1: W_1 \to W_0$. Namely, $D (\arr)$ is the annihilator in $\Sym U \otimes U^*$ of the image of $\der_1$ inside
$W_0 = \Sym U^* \otimes U$.
Recall that a hyperplane arrangement $\arr$ is called \emph{free} if $D (\arr)$ is a free $\Sym U$-module.
This notion is well studied (cf. \cite{MR1280576} and \cite[Ch. 4]{MR1217488}).
The main result of \cite{MR1280576} is that if $\arr$ is a free arrangement then the complex $W$ (and therefore also
the relation complex of $\arr$) is acyclic [ibid., Proposition 4.13 (iii)].

\subsubsection{} \label{defideals}
For a subspace $X$ of $U$ we use the notation
\begin{equation} \label{def: gen}
\gen{X}{U}:=\text{the ideal of $\Sym U$ generated by }X.
\end{equation}
Equivalently, $\gen{X}{U}$ is the ideal of all polynomials on the vector space $U^*$ vanishing on the annihilator of $X$ in $U^*$.
Then another variant of the relation complex (which is not orthogonal) arises from the
atomic datum $(\Sym U,(\gen{H^\perp}{U})_{H\in\arr})$ of graded $\Sym U$-modules.
We call this datum the \emph{defining ideals} of $\arr$. We will study the exactness of the associated minimal complex in \S \ref{algstr}.
%






\section{$L$-homotopies and \admis\ complexes} \label{reflarrangements}

\subsection{$L$-homotopies and \admiss\ sections}
In order to study the acyclicity of a minimal complex $\Orlik(\data, L)$ it will be useful to introduce a stronger notion.

\begin{definition} \label{defadmissible}
Let
\[
0\rightarrow V_n\rightarrow\dots \rightarrow V_i\xrightarrow{\der_i}V_{i-1}\rightarrow\dots\rightarrow V_0
\]
be a chain complex with compatible $L$-grading.
An \emph{$L$-homotopy} for $V$ is a contracting homotopy $d$ for $\der$,
i.e.~a sequence of morphisms $d_0:\Im\der_1\rightarrow V_1$ and $d_i:V_i\rightarrow V_{i+1}$, $i=1,\dots,n$,
satisfying
\[
d_{i-1}\der_i+\der_{i+1}d_i=\id_{V_i}, \quad i=1,\dots,n,
\]
with the additional property that
\begin{equation} \label{eqncompatibled}
d_{\rk(x)} (V_x)\subset V_{\succ x} := \dsum_{y\succ x} V_y, \quad
x\in L \setminus \{ 0 \}.
\end{equation}

We say that $V$ is \emph{\admis}, if for any $x\in L$ the complex $V_{\le x}$ admits an $L$-homotopy.
Analogously, we say that an atomic datum $\data$ is \admis\ if the minimal complex $\Orlik(\data,L)$ is \admis.
\end{definition}

For a sequence $d$ of morphisms satisfying \eqref{eqncompatibled}
and $x\in L$ we write $d_x:V_x\rightarrow V_{\succ x}$
for the restriction of $d_{\rk(x)}$ to $V_x$. Moreover,
for any $y\succ x$ we write $d_{x;y}:V_x\rightarrow V_y$ for the $y$-component of $d_x$.
Naturally, $d$ is determined by the maps $d_{x;y}$, $x\prec y\in L$.

Observe that for $\rk L \le 2$ and $R$ a field, every atomic datum over $R$ is \admis, since
in this case one may take
for $d_0$ an arbitrary section of $\der_1: \dsum_{a\in \atoms(L)}U_a \to U^L$.

\begin{remark}
The simplest examples of $L$-homotopies are the standard contracting homotopies for the Orlik-Solomon algebra $\OS(L)$,
given by multiplying by the canonical generator $e_a$ for an atom $a\in\atoms(L)$ (cf. \S \ref{reviewos}).
\end{remark}


We now describe a way to construct $L$-homotopies. The key point is to
perform the construction inductively for all complexes $V_{\le x}$, $x \in L$, at the same time.
The necessary base of the induction is given by the following notion.



\begin{definition} \label{admsection}
Let $\data=(U_0,(U_a)_{a\in\atoms})$ be an atomic datum.
A morphism $d_0:U^L\rightarrow\dsum_{a\in\atoms}U_a$ is called
a \emph{\admiss\ section} for $\data$ if the following two conditions are
satisfied:
\begin{enumerate}
\item For all $x\in L_{\irr}$
there exists an integer $h^x$, which is invertible in $R$, such that the diagram
\begin{equation} \label{cox}
\begin{CD}
U^x @>{h^x\id_{U^x}}>> U^x \\
@V{d_0}VV @AA{\der_1}A\\
{\dsum_{a\in\atoms}U_a} @>{\pi_{\le x}}>>  {\dsum_{a\le x}U_a}
\end{CD}
\end{equation}
commutes, where $\pi_{\le x}$ is the projection on the coordinates $\atoms_{\le x}$.
\item If $a,a'\in\atoms$ and $a\join a'$ is decomposable, then the $a'$-component of
$d_0\rest_{U_a}$ vanishes. In other words
\[
d_0(U_a)\subset\oplus_{a'\in\atoms(L):a\join a'\in L_{\irr}}U_{a'}.
\]
\end{enumerate}
\end{definition}

Note that these conditions imply that $\data$ is \disj. (To see this, combine the second condition with the first condition for $x = a$ and $x = a'$.)
We call the integers $h^x$, $x\in L_{\irr}$, the \emph{Coxeter numbers} of $d_0$.
(The terminology is motivated by the case of relation complexes
of Coxeter arrangements studied in \S \ref{reflarr} below, cf.~the comment after Proposition \ref{Coxprop}.)

Suppose that $d_0$ is a \admiss\ section for the atomic datum $\data=(U_0,(U_a))$
with Coxeter numbers $h^x$, $x\in L_{\irr}$. For $x\in L$
denote by $d_0^x$ the restriction of $\pi_{\le x}\circ d_0$ to $U^x$.
Then $d_0^x$ is a \admiss\ section for the atomic datum $\data_{\le x}$
with respect to $L_{\le x}$ with Coxeter numbers $h^y$, $y\in L_{\le x,\irr}$.
We also write $d_{0;a}:U^L\rightarrow U_a$ for the $a$-coordinate of $d_0$, so that
$d_0^x=\dsum_{a\in\atoms_{\le x}}d_{0;a}\big|_{U^x}$.

It is worthwhile to point out the following curious combinatorial relations among the Coxeter numbers.
\begin{lemma}
Suppose that $d_0$ is a \admiss\ section for a non-degenerate atomic datum
$(U_0,(U_a))$ with Coxeter numbers $h^x$, $x\in L_{\irr}$.
Then for any $x,z\in L_{\irr}$ with $x\le z$ we have
\[
h^z-h^x \equiv \sum_{y\in L_{\irr}:x\prec y\le z}(h^y-h^x)\pmod{n_x},
\]
where $n_x$ is the minimal positive integer such that $n_x U^x=0$, if such an integer exists, and $n_x = 0$ otherwise.
\end{lemma}

\begin{proof}
By considering $d_0^z$, we can assume that $z$ is the maximal element and that $L$ is indecomposable.
Let $x\in L_{\irr}$. The map $a\mapsto a\join x$ gives a partition
\[
\{a\in\atoms:a\not\le x\}=\coprod_{y\succ x}\{a\in\atoms:a\le y,a\not\le x\}.
\]
We therefore have $\id-\pi_{\le x}=\sum_{y\succ x}(\pi_{\le y}-\pi_{\le x})$ on $\dsum_{a\in\atoms}U_a$,
and hence
\[
d_0-d_0^x=\sum_{y\succ x}(d_0^y-d_0^x)
\]
on $U^x$. Note that if $x\prec y\in L_{\dec}$, then $y=x\disjoin a$
for some $a\in\atoms$. By the second condition of Definition \ref{admsection}, for any $a'\in\atoms_{\le x}$
we have $\pi_{\le y}\circ d_0=\pi_{\le x}\circ d_0$ on $U^{a'}$, and therefore
$d_0^y=d_0^x$ on $U^x$. Thus, we infer that
\[
d_0-d_0^x=\sum_{x\prec y\in L_{\irr}}(d_0^y-d_0^x)
\]
on $U^x$. Applying $\der_1$ on both sides we get the required equality.
\end{proof}

\begin{proposition} \label{hmtpycnstrct}
Let $\data=(U_0,(U_a))$ be an \disj\ atomic datum for a geometric lattice $L$
and let $d_0$ be a \admiss\ section for $\data$ with Coxeter numbers $h^x$, $x\in L_{\irr}$.
Let $V=\Orlik(\data,L)$.
Then there exist unique morphisms $d_{x;y}:V_x\rightarrow V_y$, $0\ne x\prec y\in L_{\irr}$, satisfying
the following property. Let
\[
d_i^z=\bigoplus_{\substack{x\in L_{\le z,i},\\ y\in L_{\le z,\irr}:\\ x\prec y}}
d_{x;y}:V_{i,\le z}\rightarrow V_{i+1,\le z},\quad i>0,
\]
and $d_0^z = \pi_{\le z}\circ d_0 |_{U^z}: U^z \rightarrow V_{1,\le z}$ as above. Then
\begin{equation} \label{indrltn}
\der_y\circ d_{x;y}=h^y\id_{V_x}-d_{\rk(x)-1}^y\circ\der_x\text{ for all pairs  }0\ne x\prec y\in L_{\irr},
\end{equation}
and for any $z\in L_{\irr}$,
$\frac1{h^z}d^z$ is an $L$-homotopy for $V_{\le z}=\Orlik(\data_{\le z},L_{\le z})$.
\end{proposition}

\begin{proof}
First, recall that $V$ is supported on indecomposables by Lemma \ref{irred}.
The uniqueness of the maps $d_{x;y}$ is clear by induction on $\rk(x)$, since $\der_y$ is injective for all $y$.

We will prove the existence of $d_{x;y}$ and the homotopy property of the ensuing maps $\frac1{h^z}d^z$
by induction on $\rk (x)$. Namely,
suppose that for some $i>0$ there exist maps $d_{x;y}: V_x \rightarrow V_y$ for all pairs
$0\ne x\prec y\in L_{\irr}$, $\rk (x) < i$,
that satisfy \eqref{indrltn} as well as the relations
\begin{equation} \label{indhmtpy}
\der_{j+1}d_j^z+d_{j-1}^z\der_j|_{V_{\le z,j}}=h^z\id_{V_{\le z,j}}
\end{equation}
for all $z\in L_{\irr}$ and $0<j<i$.
We will show that we can define $d_{x;y}$ for all $x \prec y \in L_{\irr}$, $\rk (x) = i$, such that
\eqref{indrltn} is again satisfied and \eqref{indhmtpy} continues to hold for $j=i$.

First note that for $i > 1$ we may apply \eqref{indhmtpy} for $j=i-1$ and compose it with $\der_i$
to obtain
\begin{equation} \label{j=i-1}
\der_id_{i-1}^z\der_i|_{V_{\le z,i}}=(\der_id_{i-1}^z+d_{i-2}^z\der_{i-1})\der_i|_{V_{\le z,i}}=h^z\der_i|_{V_{\le z,i}}
\end{equation}
for all $z\in L_{\irr}$. By \eqref{cox}, this relation is also applicable in the case $i=1$.
In particular, given a pair $x\prec y \in L_{\irr}$ with $\rk (x) = i$, by setting $z=y$ we obtain
\[
\der_{\prec y}\circ d_{i-1}^y\circ\der_x=
\der_i\circ d_{i-1}^y\circ\der_x=h^y\der_x.
\]
In other words,
$h^y\id_{V_x}-d_{i-1}^y\circ\der_x$ maps $V_x$ to $\Ker\der_{\prec y} \subset V_{\prec y}$.
By \eqref{defmincmplx} (applied to $y$), it follows that there exists a map $d_{x;y}:V_x\rightarrow V_y$
satisfying \eqref{indrltn}.

It remains to show that
\begin{equation} \label{indstephmtpy}
\der_{i+1}d_i^z+d_{i-1}^z\der_i|_{V_{\le z,i}}=h^z\id_{V_{\le z,i}},\ \ \ z\in L_{\irr}.
\end{equation}
To that end we claim that the left-hand side of \eqref{indstephmtpy} preserves $V_x$
for any $x\in L_{\le z,i}$. Granted this claim, since $\der_i$ is injective on each $V_x$,
\eqref{indstephmtpy} follows from the relation
\[
\der_i(\der_{i+1}d_i^z+d_{i-1}^z\der_i|_{V_{\le z,i}})=h^z\der_i|_{V_{\le z,i}},
\]
which in turn follows from \eqref{j=i-1}.

To prove the claim, fix $x\in L_{\le z,i}$ and write $\der_x=\dsum_{x'\prec x}\der_{x;x'}$
where $\der_{x;x'}:V_x\rightarrow V_{x'}$.
Then the restriction of the left-hand side of \eqref{indstephmtpy} to $V_x$ is
\begin{multline*}
\sum_{y\in L_{\le z,\irr}: y \succ x}\der_yd_{x;y}+d_{i-1}^z\der_x=
\sum_{y\in L_{\le z,\irr}: y \succ x} (h^y\id-d_{i-1}^y\der_x)+d_{i-1}^z\der_x=\\
\sum_{y\in L_{\irr}:z\ge y\succ x}\big[h^y\id-\sum_{x',y'\in L_{\irr}:x'\prec x,
y\ge y'\succ x'}d_{x';y'}\der_{x;x'}\big]
+\sum_{x',y'\in L_{\irr}:x'\prec x,z\ge y'\succ x'}
d_{x';y'}\der_{x;x'}.
\end{multline*}
The contribution from any pair $y'\succ x'$ with
$y'\ne x$ to the last sum cancels with its contribution to the
middle sum for $y=x\join y'$. Observe here that if $x \join y' \in L_{\dec}$ then
by Lemma \ref{irrdecom} we have $x'= x \meet y' = 0$, i.e.~$i=1$ and $x\in\atoms_{\le z}$, and that in this case
$d_{0;y'}\der_x=0$, since $d_0^{y'}\rest_{U_x}=0$ by the second condition of Definition \ref{admsection}.
Therefore, only the terms with $y'=x$ contribute, and we conclude that
$\der_{i+1}d_i^z+d_{i-1}^z\der_i$ preserves $V_x$, as claimed.
\end{proof}




For the record, we point out the following variant of Proposition \ref{hmtpycnstrct}, which is proved
in exactly the same way.
\begin{proposition}
Let $\data=(U_0,(U_a))$ be an atomic datum for a geometric lattice $L$.
Let $d_0:U^L\rightarrow\dsum_{a\in\atoms}U_a$ be a morphism with the property that
for any $x\in L$ there exists an integer $h^x$ such that the diagram
\eqref{cox} commutes.
Let $V=\Orlik(\data,L)$.
Then there exist unique morphisms $d_{x;y}:V_x\rightarrow V_y$, $0\ne x\prec y\in L$, satisfying
the following property. Let
\[
d_i^z=\dsum_{x\in L_{\le z,i},y\in L_{\le z}: x\prec y}d_{x;y}:V_{i,\le z}\rightarrow V_{i+1,\le z},\quad i>0,
\]
and $d_0^z = \pi_{\le z}\circ d_0 |_{U^z}: U^z \rightarrow V_{1,\le z}$ as above. Then
\[
\der_y\circ d_{x;y}=h^y\id_{V_x}-d_{\rk(x)-1}^y\circ\der_x\text{ for all pairs }0\ne x\prec y\in L.
\]
Moreover, for any $z\in L$ such that $h^z$ is invertible in $R$,
$\frac1{h^z}d^z$ is an $L$-homotopy for $V_{\le z}=\Orlik(\data_{\le z},L_{\le z})$.
\end{proposition}

\begin{remark} \label{remarkos}
In particular, this applies to atomic data with constant coefficients (cf. \S \ref{constantcoefficients}),
where for a fixed $a\in\atoms(L)$ we let $d_0:U_0 \rightarrow U_a$ be the identity map.
In this case $h^x=1$ if $a\le x$ and $h^x=0$ otherwise.
The ensuing contracting homotopy for the Orlik-Solomon algebra $\Orlik (\data, L) = \OS (L) \otimes U_0$ is the standard one obtained
by multiplying in $e_a$ (cf. \S \ref{reviewos}).
\end{remark}

\subsection{Complex reflection arrangements} \label{reflarr}
Now let $V$ be a complex vector space and $G$ a finite (complex) reflection group on $V$.
Choose a $G$-invariant inner product $(\cdot,\cdot)$ on $V$.
Consider the reflection arrangement $\arr = \{H_a\}_{a\in\atoms}$
consisting of the reflecting hyperplanes $H_a$ of $G$ (cf.~\cite[Ch.~6]{MR1217488})
and let $L$ be the corresponding intersection lattice.
Let $\data=(V,(H_a^\perp))$ be the defining atomic datum of the reflection arrangement,
where we identified $V$ and $V^*$ through $(\cdot,\cdot)$.
For each $a\in\atoms$ choose a vector $w_a\in H_a^\perp$ such that $(w_a,w_a)=1$.
Note that $L$ is indecomposable if and only if $G$ is indecomposable (or equivalently, irreducible).
Also, for any $x\in L$, the lower interval $L_{\le x}$ is the intersection lattice of the reflecting
hyperplanes of a reflection subgroup $G_x$ of $G$, namely the pointwise stabilizer of $x$ \cite[Corollary 6.28]{MR1217488}.

\begin{proposition} \label{Coxprop}
Under these assumptions
$d_0(v)=(2(v,w_a)w_a)_{a\in\atoms}$, $v \in V$, is a \admiss\ section for the defining atomic datum $\data=(V,(H_a^\perp))$
of the reflection arrangement $\arr$ with Coxeter numbers
$h^x=\frac{2\card{\atoms_{\le x}}}{\rk(x)} \in\Z_{>0}$, $x\in L_{\irr}$.
\end{proposition}

Note that if $G$ is a real reflection group (i.e.~a finite Coxeter group)
then $h^x$ is the Coxeter number of $G_x$ in the usual sense (cf.~\cite[p.~257]{MR1217488}).
We remark that for arbitrary complex reflection groups a different notion of Coxeter number, which
appears naturally in other contexts, can be found in the literature (cf.~\cite{Grif-Gor}).\footnote{We thank the referee for pointing this out to us.}

\begin{proof}
In Definition \ref{admsection}, the commutativity of \eqref{cox} and the integrality of $h^x$ follow from \cite[Proposition 6.93]{MR1217488}
and [ibid., Theorem 6.97], respectively.
More generally, if $x=\disjoin_{i=1}^kx_i$ is the decomposition of $x \in L$ into
indecomposables, then $U^x=\dsum_{i=1}^k U^{x_i}$ is
an orthogonal direct sum with respect to $(\cdot,\cdot)$ and therefore
\begin{equation} \label{Coxsection}
\der d_0^x(\sum_{i=1}^kv_i)=\sum_{i=1}^kh^{x_i}v_i \ \text{ for } v_i\in U^{x_i}, \, i = 1, \dots, k.
\end{equation}
The second condition of Definition \ref{admsection} follows from the fact that if $a\join a'$
is decomposable then the vectors $w_a$ and $w_{a'}$ are orthogonal.
\end{proof}

Let $x\in L$ and $\proj_x$ be the orthogonal projection $V\rightarrow
x=V/x^\perp$. The restriction $\arr_{\ge x}$ of the reflection arrangement $\arr$ to $x$, which is defined by
the vectors $\proj_x (w_a)$, $a \not\le x$, is in
general no longer a reflection arrangement (cf.~\cite[Example 6.83]{MR1217488} for a simple example).
Set
\[
d_{0,\ge x}(v)=\big(\sum_{a'\in\atoms:a'\join x=a}
2(v,w_{a'})\proj_x(w_{a'})\big)_{a\in\atoms(L_{\ge x})},\ \ \ v\in x.
\]
We have $\der\circ d_{0,\ge x}=\proj_x\der\circ d_0\rest_{x}$.

\begin{proposition}
$d_{0,\ge x}$ is a \admiss\ section for the
defining atomic datum of the restricted hyperplane arrangement $\arr_{\ge x}$ on $x$.
\end{proposition}

\begin{proof}
Let $y\in L_{\ge x}$ and $y=\disjoin_{i=1}^k y_i$ its decomposition into
indecomposables in $L$. By Lemma \ref{irrdecom}, we have
$x=\disjoin_{i=1}^k x_i$ with $x_i \le y_i$.
For $v=\sum_{i=1}^k v_i\in U^y \cap x$ with $v_i\in U^{y_i}$, we get from \eqref{Coxsection} that
\[
\der d_{0,\ge x}^y(v)=\proj_x\der d_0^y(v)=\sum_{i=1}^k h^{y_i}\proj_xv_i.
\]
Suppose now that $y$ is indecomposable in $L_{\ge x}$. This
is equivalent to $x_i = y_i$ for all but a single index $i$, say for all $i \neq 1$.
Therefore, $v_i \in U^{y_i} \subset U^x = x^\perp$ for all $i \neq 1$, and
$\der d_{0,\ge x}^y(v)=h^{y_1}\proj_xv_1=h^{y_1}v$.

To check the second condition of Definition \ref{admsection}, suppose that
$a_1$, $a_2$ are distinct atoms of $L_{\ge x}$ and that $a_1 \join a_2$ is decomposable in $L_{\ge x}$.
This is equivalent to $\atoms_{\le a_1 \join a_2} = \atoms_{\le a_1} \cup \atoms_{\le a_2}$. Set
$A_1 = \atoms_{\le a_1} \backslash \atoms_{\le a_2}$ and $A_2 = \atoms_{\le a_2} \backslash \atoms_{\le a_1}$ and let $U_i \subset U^{a_1 \join a_2}$
be the span of the vectors $w_z$, $z \in A_i$ ($i = 1, 2$). We have then
\[
A_i = \{ a \in \atoms_{\le a_1 \join a_2} : \proj_x (w_a) \in U^{a_i} \cap x, \, \proj_x (w_a) \neq 0  \}.
\]
Observe now that
$(w_{z_1},w_{z_2}) = 0$
for all $z_1 \in A_1$ and $z_2 \in A_2$.
For if we had $(w_{z_1},w_{z_2}) \neq 0$, then a complex reflection $s \in G_{a_1 \join a_2}$ with fixed hyperplane $H_{z_2}$
would satisfy $s (z_1) \notin \atoms_{\le a_1} \cup \atoms_{\le a_2}$,
in contradiction to
$\atoms_{\le a_1 \join a_2} = \atoms_{\le a_1} \cup \atoms_{\le a_2}$.

Since the group $G_{a_1 \join a_2}$ is generated by its subgroup $G_x$
together with complex reflections with fixed hyperplanes $H_z$, $z \in A_1 \cup A_2$, we conclude that
$G_{a_1 \join a_2}$ stabilizes the mutually orthogonal subspaces $U_1$ and $U_2$. Therefore we get a decomposition
$a_1 \join a_2 = y_1 \disjoin y_2$ such that $a_1 = b_1 \disjoin y_2$ and $a_2 = y_1 \disjoin b_2$ with $b_i \prec y_i$, and
consequently $x = b_1 \disjoin b_2$. Since $U^{y_1}$ and $U^{y_2}$ are orthogonal, $U^{a_1} \cap x \subset U^{y_2}$ and
$w_{a'} \in U^{y_1}$ for $a' \in A_2$, we obtain that the $a_2$-component
of $d_{0,\ge x}(v)$ is zero for all $v \in U^{a_1} \cap x$.
\end{proof}

Since any finite complex reflection group is a product of
indecomposable groups,
we proved:

\begin{theorem} \label{roots}
Any restriction of a reflection arrangement is \admis.
\end{theorem}

\begin{remark}
Recall that by \cite{MR1280576} every free hyperplane arrangement is $n$-formal.
It is also known that restrictions of Coxeter arrangements are always free \cite{MR1231562}.
Therefore the relation complex of the restriction of a Coxeter arrangement is acyclic.
However, the argument in \cite{MR1280576} is indirect and
does not seem to produce a contracting homotopy (let alone an $L$-homotopy) explicitly.
Also, the proof of freeness in \cite{MR1231562} involves a case by case analysis
to deal with restrictions to elements of rank greater than one.
An alternative, classification-free proof for the freeness of restrictions of Coxeter arrangements was obtained
in \cite{MR1722107} using the algebraic representation theory of reductive groups.
Our approach to the relation complex is direct and elementary, but does not touch the question of freeness.
\end{remark}

\begin{remark}
It is instructive to explicate the relation complexes of the various reflection arrangements.
The case of rank two considered in Example \ref{examplerank2} is essentially trivial.
For any hyperplane arrangement $\arr$ of rank two, we can construct a \admiss\ section for its defining atomic datum as follows.
For distinct atoms $a$, $b \in \atoms$ let $\proj_{a,b}:V\rightarrow H_a^\perp\dsum H_b^\perp$
be the inverse map to the linear isomorphism $\der$ between the two-dimensional vector spaces
$H_a^\perp\dsum H_b^\perp$ and $V$.
Set $d_0=\sum_{\{a,b\} \subset \atoms, \, a \neq b}\proj_{a,b}$.
Here $h={\card{\atoms}\choose2}$ and $h^a=\card{\atoms}-1$ for each atom.
In the appendix we will consider the infinite families of
indecomposable Coxeter arrangements.
\end{remark}

\section{A generalization of the Orlik-Solomon algebra} \label{algstr}

In this section we consider a general construction, modeled after the Orlik-Solomon algebra, which
allows us to pass from a \admis\ complex of vector spaces $V$ to a \admis\ complex of modules
over the symmetric algebra $\Sym (V_0)$.
We then apply this construction to truncations of a lattice $L$ and more specifically to the relation complexes
of \S \ref{exmplatomic}.

Suppose that we are given a chain complex $(V=\dsum_{x \in L} V_x,\der)$ of vector spaces
over a field $K$ of characteristic zero with a compatible grading
by a geometric lattice $L$.
Let $S=S(V)$ be the universal supercommutative algebra generated by $V$,
i.e., the quotient of the tensor algebra of $V$ by the two-sided ideal
generated by all expressions $uv-(-1)^{ij}vu$ for $u\in V_i$, $v\in V_j$, $i,j=0,\dots,n$.
Thus, $S\simeq\Sym(\dsum_{m\text{ even}}V_m)\otimes\bigwedge(\dsum_{m\text{ odd}}V_m)$.
We grade $S$ by assigning degree $i$ to $V_i$.
By extending $\der$ to a (super-)derivation $\tilde\der$ on $S$,
we obtain a differential graded algebra $(S,\tilde\der)$.
Note that $S_0$ is the symmetric algebra of $V_0$ and that $S$ is an algebra over $S_0$.

The algebra $S$ carries a canonical grading by the lattice $L$, in which $V_x \subset S_x$
and
\[
a\in S_x, b\in S_y \quad \text{implies} \quad ab\in S_{x\join y} \quad \text{for } x,y\in L.
\]
Explicitly, let $\basis_x$ be a basis of $V_x$, $x\in L$, set
$\basis_{\even}=\coprod_{x\in L_{\even}}\basis_x$ and $\basis_{\odd}=\coprod_{x\in L_{\odd}}\basis_x$.
Then, as $\{u_1,\dots,u_k\}$ ranges over all subsets of $\basis_{\odd}$ and $\{v_1,\dots,v_m\}$ ranges
over all multisets (i.e. sets with multiplicities) of elements of $\basis_{\even}$, the products
\[
v_1\dots v_mu_1\dots u_k
\]
(defined up to sign if $k>1$) form a basis of $S$, and if $v_i\in V_{x_i}$, $u_j\in V_{y_j}$, then we have $v_1\dots v_mu_1\dots u_k\in S_x$ for
$x=x_1\join\dots\join x_m\join y_1\join\dots\join y_k$.
Note that the notation $S_0$ is unambiguous and that for any $a\in\atoms (L)$ the component
$S_{a}$ is mapped by $\tilde\der$ onto the ideal $\gen{\der V_a}{V_0}$ of $S_0$ (see \eqref{def: gen} for the notation).

The canonical $L$-grading on $(S,\tilde\der)$ is \emph{not} compatible in the sense of Definition \ref{compatible}.
The problem is that if $r_i\in V_{x_i}$, $i=1,\dots,m$, then
the product $r_1 \cdots r_m$ is of degree $\sum\rk (x_i)$, which differs from $\rk (\join x_i)$
if $x_1,\dots,x_m$ are dependent (cf.~Definition \ref{dependentdefinition}).
To rectify this, consider the vector space $I \subset S$ spanned by all products
\[
r_1\cdots r_m, \quad r_i\in V_{x_i},\ i=1,\dots,m,\ x_1,\dots,x_m\in L\text{ dependent}.
\]
Observe that as a consequence of Lemma \ref{deptriv}, the space $I$ is a two-sided ideal of $S$
and hence $\ideal=\ideal(L,V) := I+\tilde\der(I)$ is a differential ideal of $S$
(i.e.~a graded ideal which is mapped to itself under $\der$).
Also, $I$ is graded with respect to both the degree and the lattice $L$, and therefore $\ideal$ is graded
with respect to the degree.
We can now define the main object of this section.


\begin{definition}
The \emph{generalized Orlik-Solomon algebra} of the $L$-graded complex $V$ is
the differential graded algebra $\alg=\alg(L,V):=S (V)/\ideal (L,V)$.
\end{definition}

\begin{proposition} \label{graded}
The generalized Orlik-Solomon algebra $\alg = \alg (L,V)$ is compatibly $L$-graded.
We have $\alg_0=S_0 = \Sym (V_0)$ and for any $a\in\atoms(L)$ the image $\tilde\der(\alg_{a})$
is the ideal $\gen{\der V_a}{V_0}$ of $S_0$ generated by $\der V_a$.
\end{proposition}

\begin{proof}
We need to show that the ideal $\ideal = \ideal (L,V)$ is $L$-graded.
Let $r_i\in V_{x_i}$, $i = 1, \dots, m$, where $x_1,\dots,x_m \in L$ are dependent.
It follows from part \ref{precdep} of Lemma \ref{deptriv} that we can write
\[
\tilde\der(r_1 \cdots r_m)=u+v
\]
where $u\in I_{\prec \join x_i}$ and $v\in S_{\join x_i}$. In particular, we have $v\in \ideal \cap S_{\join x_i} = \ideal_{\join x_i}$.
Since $\ideal$ is spanned as a vector space by the $L$-graded ideal $I$ and the elements
$\tilde\der(r_1 \cdots r_m)$,
it follows that $\ideal$ is $L$-graded and that for all $x\in L$ we have
\begin{equation} \label{nozliga}
\tilde\der(I_x)\subset\ideal_x+I_{\prec x}.
\end{equation}
Hence, $\alg$ is $L$-graded. It is clear from the definition of $I$ that the grading is compatible.
The last part of the proposition is also clear, since $\ideal_0=0$.
\end{proof}

\begin{remark}
Let $V_a=V_0$ for all $a\in\atoms(L)$ and $V_x=0$ otherwise, and
consider the complex $V$ with $\der_a=\id$.
Recall the construction of the Orlik-Solomon algebra $\OS (L) = \extalg / \ideal$ in \S \ref{reviewos} above.
Then the algebra $S(V)$ can be identified with $\extalg \otimes S_0$ and $\ideal (L,V)$ with $\ideal \otimes S_0$.
Therefore $\alg (L,V)=\OS(L)\otimes S_0$ as differential graded algebras.
\end{remark}

\begin{proposition} \label{main2}
Suppose that $V$ is a \admis\ $L$-graded complex of vector spaces. Then $\alg(V)$ is also \admis\ (as
a complex of vector spaces).
In particular, if in addition the maps $\der_1|_{V_a}$, $a \in \atoms$, are all injective, then
the minimal complex
\[
\Orlik(\Sym V_0,(\gen{\der V_a}{V_0})_{a \in \atoms (L)}) \simeq \alg(V)
\]
is exact.
\end{proposition}

\begin{proof}
For any $x \in L$ we have
$S(V)_{\le x}=S(V_{\le x})$
and therefore
$\alg(V)_{\le x}=S(V_{\le x})/\ideal_{\le x}$.
Consider the complex $V'_{\le x}$ which coincides with $V_{\le x}$ in degrees at least one and
has $V'_{\le x, 0} = \der(V_{\le x, 1})$.
Also, let $W_0$ be a complement to $\der (V_{\le x, 1})$ in $V_0$. Then it is easy to see
that $\alg (V_{\le x}) \simeq \Sym W_0 \otimes S (V'_{\le x}) / \ideal_{\le x} (V'_{\le x})$ as
differential graded algebras. Therefore, it is enough to show the existence of
an $L$-homotopy on $\alg (V'_{\le x}) = S (V'_{\le x}) / \ideal_{\le x} (V'_{\le x})$.
Let $d$ be an $L$-homotopy for $V_{\le x}$.
From $d$ we can construct a (super-)derivation $\tilde d$ of
$S(V'_{\le x})$. Consider the alternative grading on this algebra obtained by
assigning degree one to every element of $V'_{\le x}$.
Then clearly $\tilde\der$ and $\tilde d$ preserve this grading,
and on the degree $n$ part $S (V'_{\le x})^{(n)}$ of $S (V'_{\le x})$ we have
$\tilde d \tilde\der + \tilde \der \tilde d = n \id$.
Setting $d' = n^{-1} \tilde d$ on $S (V'_{\le x})^{(n)}$, $n \ge 1$, we therefore obtain
a contracting homotopy for $\tilde\der$ on $S (V'_{\le x})$ (which is only defined on $\im\tilde\der_1$ at
the $0$-th position).
Observe that by \eqref{nozliga} we have $\ideal_{\le x}=I_{\le x}+
\tilde\der(I_{\le x})$.
Since $d$ is an $L$-homotopy, we conclude from Lemma \ref{deptriv}
that $\tilde d$ and $d'$ preserve $\ideal_{\le x}$.
It follows that $d'$ defines an $L$-homotopy for $\alg(V'_{\le x})$.
\end{proof}

We have the following immediate application to hyperplane arrangements.

\begin{corollary} \label{CorDefideals}
Suppose that $L$ is the intersection lattice of a hyperplane
arrangement $\arr$ in a vector space $U^*$. Assume that the relation complex
$V$ of $\arr$ is \admis. Then the minimal complex associated to the defining ideals
of $\arr$ (cf. \S \ref{defideals}) is exact.
\end{corollary}

Further cases can be derived by applying the truncation maps of \S\ref{truncation}.

\begin{lemma} \label{LemmaTruncAdm}
Let $V$ be a \admis\ $L$-graded complex, and $l: L_{\ge k} \to \Lambda$
a non-degenerate $k$-th truncation map. Then the truncated shifted complex
$V^{(k)} = (V_{i+k})_{i \ge 0}$ with the natural $\Lambda$-grading
$V^{(k)}_\lambda := \dsum_{x \in L: l(x) = \lambda} V_x$, $\lambda \in \Lambda$, is
a \admis\ $\Lambda$-graded complex.
\end{lemma}

\begin{proof}
By our assumption on $l$, the complex $V^{(k)}$ is compatibly $\Lambda$-graded.
Let $\tilde x \in \Lambda$ and let $x_1,\dots,x_m$ be the maximal elements of $L$
with $l (x) \le_\Lambda \tilde x$.
It follows from Lemma \ref{trunclemma} that the obvious map of complexes
\[
\dsum_{i=1}^m V^{(k)}_{\le x_i} \longrightarrow V^{(k)}_{\le \tilde x}
\]
is an isomorphism except at the lowest point (corresponding to $V_k$),
where we only have an inclusion of the left-hand side into the right-hand
side. It follows that if $d^i$ is an $L$-homotopy for $V_{\le x_i}$
then $d = \oplus_{i=1}^m d^i$ gives an $L$-homotopy of $V^{(k)}_{\le \tilde x}$.
\end{proof}

In the case of hyperplane arrangements and relation spaces, Lemma \ref{LemmaTruncAdm} and Proposition \ref{main2}
imply the following result.

\begin{corollary} \label{CorProj}
Let $\arr$ be a hyperplane arrangement in the space $U^*$ and $V$ the
relation complex of $\arr$. Let $P$ be a subspace of $U$ of codimension $k$ in general position
with respect to $\arr$ (in the sense of Definition \ref{generalposition}).
Let $\Lambda$ be the intersection lattice
of the hyperplane arrangement in $P^*$ given by the hyperplanes $\pi_P (x) \subset P^*$, $x \in L_{k+1}$,
and $l: L_{\ge k} \to \Lambda$ the associated truncation map.
If $V$ is \admis, then the minimal $\Lambda$-graded complex
$\Orlik (\Sym V_k, (\gen{\der V_{l^{-1} (a)}}{V_k})_{a \in \atoms (\Lambda)})$ is exact.
\end{corollary}

Recall here that because of the general position assumption on $P$ the truncation map $l$ gives a bijection between $L_{k+1}$ and $\atoms (\Lambda)$.




\section{An exact sequence of Bernstein-Lunts type} \label{zonotope}

We now turn to hyperplane arrangements over the real numbers and the
associated cone decompositions (polyhedral fans) together with their dual polytopes. Our main objects of study are certain
modules defined by the underlying graphs of these polytopes.

\subsection{Hyperplane arrangements and zonotopes}
First, we quickly review the duality between hyperplane arrangements and zonotopes (cf.~\cite[Ch.~7]{MR1311028}
for more details).

Let $U$ be an $n$-dimensional real vector space, $U^*$ its dual space, and $\arr$ a hyperplane arrangement in $U^*$
with intersection lattice $L$.
Then $\arr$ induces a decomposition of the space $U^*$ into convex polyhedral cones,
namely the closures of the
connected components of $x \setminus \cup_{H\in\arr: x \not\subset H} (x \cap H)$ for $x\in L$.
The set $\Sigma (\arr)$ of these cones forms a lattice with the partial order given by
reverse inclusion, and
the relative interiors of the cones in $\Sigma (\arr)$ form a partition of the set $U^*$.
For any $0 \le k \le n$ we write $\Sigma_k (\arr)$ for the subset of cones
of codimension $k$.
There is a natural lattice map $\Sigma (\arr) \to L$ which associates to each cone $C$ its linear span inside
the vector space $U^*$.
Dually, we consider the zonotope (Minkowski sum of line segments)
$\zono = \sum_{H \in \arr} [-1,1] u_H \subset U$, where $u_H$ is an arbitrary non-zero vector
in the one-dimensional space $H^\perp \subset U$. Clearly, $\zono$ is a convex polytope in the subspace $U'$ of $U$ spanned by
the lines $H^\perp$, $H \in \arr$.
Since the vectors $u_H$ are only unique up to scalar multiplication, the zonotope
$\zono$ is \emph{not} determined by the hyperplane arrangement $\arr$, even up to affine equivalence. However, $\zono$ is evidently determined up to combinatorial equivalence
by $\arr$.

By mapping a cone $C \in \Sigma (\arr)$ to the face
\[
F = \{ u \in \zono: \sprod{c}{\cdot} \text{ attains its maximum value
on $\zono$ at $u$} \}
\]
of $\zono$, where $c$ is an arbitrary vector in the relative interior of $C$, we obtain a lattice isomorphism
between $\Sigma (\arr)$ and the face lattice of $\zono$.
Under this isomorphism the dimensions of $C$ and $F$ satisfy $\dim C+\dim F=n$.
The induced lattice map from the faces of $\zono$ to the intersection lattice $L$
associates to a face $F$ the space $x=x (F) = (F+(-F))^\perp \in L$, i.e., the annihilator in $U^*$ of the vector part of $F$.
We say that $F$ is of type $x$ in this case. Thus, two faces are of the same type if and only if their affine hulls are parallel.
We regard the $1$-skeleton of $\zono$ as a graph where each edge is labeled
by the corresponding atom of $L$.

Let $\facets_k$ be the set of $k$-dimensional faces of $\zono$, $k=-1,\dots,n$, where
we set $\facets_{-1}=\{\emptyset\}$. By the above, for any $k \ge 0$ the set $\facets_k$ is in bijection with
the set $\Sigma_k (\arr)$.
For any $x\in L$ the image $\zono_{\ge x}$ of $\zono$ under the projection modulo $x^\perp$ is a zonotope dual
to the restricted arrangement $\arr_{\ge x}$ in the space $x$. For $i = \rk (x), \dots, n$ the faces of $\zono_{\ge x}$ of dimension
$i-\rk (x)$ are the projections of the $i$-dimensional faces of $\zono$ of type $\ge x$. We denote the set of
these faces by $\facets_i^{\ge x}$.

\begin{remark}
For crystallographic Coxeter arrangements (and their restrictions) the combinatorial structure of the
polyhedral fans $\Sigma (\arr)$ and zonotopes $\zono$ is parallel to that
of parabolic subgroups of reductive algebraic groups (cf.~\cite[\S 2.4]{MR2811597}).
If one chooses the roots for the vectors $u_H$ then one obtains a zonotope $\zono$
which is symmetric under the corresponding Weyl group. The particular case of the root system of type $A_n$ yields the well-known
regular permutahedron (cf.~\S \ref{higherbruhat} below, \cite[pp. 17-18, 200]{MR1311028}, \cite{MR2487491}).
\end{remark}

\subsection{Graphical modules and Euler-Poincar\'e complexes} \label{graphicalmodules}
Let $R$ be a ring with $1$.
Consider the following abstract construction: given a graph $\Gamma = (V,E)$, an $R$-module $M_0$
and for each edge $e$ of $\Gamma$ a submodule $M_e$ of $M_0$, let $\module$ be the $R$-module of all $M_0$-valued
functions on the vertex set $V$ of $\Gamma$ satisfying congruence conditions modulo $M_e$ along the edges, i.e.
\[
\module = \{ m: V \to M_0:  m (v) - m (v') \in M_{e} \text{ for all }
e = \{ v, v' \} \in E\}.
\]
This construction has been studied by Guillemin and Zara \cite{MR1701922, MR1823050, MR1959894} motivated by
applications to equivariant cohomology (cf.~\cite{MR1489894}).

We will only study the case where $\Gamma$ is the $1$-skeleton of a zonotope $\zono$ and $M_e=M_{e'}$ for parallel edges $e$ and $e'$ of $\zono$.
Until the end of \S \ref{graphicalmodules} let $\arr$ be a hyperplane arrangement and let $\zono$ be a dual zonotope as above.
\begin{definition}
Let $\data=(M_0,(M_a))$ be an atomic datum of $R$-modules with respect to the intersection lattice $L$ of $\arr$.
The \emph{graphical module} of the zonotope $\zono$ with respect to the atomic datum $\data$ is
\[
\module=\module(\zono,\data) = \{ m: \facets_0 \to M_0:  m (v) - m (v') \in M_{x(e)} \text{ for }
v, v' \prec e \in \facets_1\}.
\]
\end{definition}

An important motivating example will be given in Remark \ref{PiecewisePolynomials} below.
In \S \ref{SubsectionModules} we treat some new cases of this construction.

Following the work of Bernstein-Lunts,
we consider a complex relating the graphical module $\module(\zono,\data)$ to the entire face lattice of $\zono$.
For this purpose recall the construction of
the \emph{Euler-Poincar\'e complex} of the zonotope $\zono$ with coefficients in an $R$-module $M_0$. It is defined as follows.
Fix an orientation of the real vector space $U$. Set
\[
(EP): 0\rightarrow M_0^{\facets_{-1}}\xrightarrow{\bnd_0}M_0^{\facets_0}\xrightarrow
{\bnd_1}M_0^{\facets_1}\xrightarrow{\bnd_2}\dots\xrightarrow{\bnd_n}
M_0^{\facets_n} \rightarrow 0,
\]
where the boundary maps are
\[
\bnd_j((m_{F'})_{F'\in\facets_{j-1}})_F=\sum_{F' \prec F}\sign (F',F) m_{F'}, \quad F\in\facets_j,
\]
for $j=0,\dots,n$, and the signs $\sign (F',F) \in \{ \pm 1 \}$ are determined by the chosen orientation of $U$.
It is well-known that this complex is exact.

We now define a subcomplex $M_{\data}^\zono$ of the Euler-Poincar\'e complex $EP$, which is graded by the face lattice of $\zono$
in ranks at least one and whose rank zero piece is the graphical module $\module (\zono,\data)$ instead of $M_0^{\facets_0}$.

\begin{definition}
Let $\arr$ be a hyperplane arrangement and $\zono$ be a dual zonotope.
Let $\data=(M_0,(M_a))$ be an atomic datum of $R$-modules with respect to the intersection lattice $L$ of $\arr$.
For any face $F$ of $\zono$ set
\begin{equation} \label{def: MF}
M_F=M^{x(F)}= \sum_{a\in\atoms(L):a \le x(F)} M_a \subset M_0.
\end{equation}
Then the \emph{modified Euler-Poincar\'e complex} of $\zono$ with respect to the atomic datum $\data$ is given by
\begin{multline*}
(M_{\data}^\zono):
0\rightarrow M_0\xrightarrow{\bnd_0}\module\xrightarrow{\bnd_1}
\dsum_{F\in\facets_1}M_F\xrightarrow{\bnd_2}
\dsum_{F\in\facets_2}M_F\xrightarrow{\bnd_3}\dots\xrightarrow{\bnd_n}M_{\zono}
\rightarrow 0.
\end{multline*}
\end{definition}

Note that by the exactness of $EP$ we have a short exact sequence
\begin{equation} \label{exct1}
0\rightarrow M_0\rightarrow\module\rightarrow\Ker\bnd_2\rightarrow0,
\end{equation}
in other words, the complex $M_{\data}^\zono$ is exact at the first two places.
We are interested in the question whether $M_{\data}^\zono$ is actually exact at all places. In this
case, it provides a resolution for $\module$ in terms of the modules $M_0 / M_F$.
Indeed, by considering the cokernel of the natural inclusion of complexes $M_{\data}^\zono \hookrightarrow EP$, the exactness of
$M_{\data}^\zono$ is equivalent to the exactness of
\[
0\rightarrow \module\rightarrow
\dsum_{F\in\facets_0}M_0\rightarrow
\dsum_{F\in\facets_1}M_0 / M_F\rightarrow
\dsum_{F\in\facets_2}M_0 / M_F\rightarrow\dots\rightarrow M_0 / M_{\zono}
\rightarrow 0.
\]

The following criterion allows us to reduce the problem of the exactness of the modified Euler-Poincar\'e complex to the
exactness of the minimal complexes $\Orlik (\data,L)$ for the intersection lattice $L$, which were studied in the previous part
of our paper.

\begin{proposition} \label{main}
Let $\data$ be an atomic datum of $R$-modules with respect to the intersection lattice $L$ of a hyperplane arrangement $\arr$ with
dual zonotope $\zono$.
Let $N=\Orlik(\data,L)$ with differential $\der$ be the minimal complex associated to $\data$, and suppose that for every $x\in L$ the complex
$(N_{\le x},\der)$ is exact.
Then the modified Euler-Poincar\'e complex $M_{\data}^\zono$ is exact.
\end{proposition}

\begin{proof}
In the course of the proof we will only consider complexes indexed by the non-negative integers.
For any $i=0,\dots,n$ and any face $F$ of $\zono$ set
\[
N_i^F=\begin{cases}\dsum_{y\le x(F):\rk(y)=i}N_y,&i>0,\\
M_F = M^{x(F)},&i=0.\end{cases}
\]
By assumption, for any $F$ the complex $((N_i^F),\der)$, where $\der_0=0$, is exact at all places,
i.e.~for every $F$ and $i$ we have a short exact sequence
\begin{equation} \label{Nexact}
0\rightarrow\Ker\der_i\rest_{N_i^F}\rightarrow N_i^F\xrightarrow{\der_i}
\Ker\der_{i-1}\rest_{N_{i-1}^F}\rightarrow0.
\end{equation}

For any $i = 0,\dots,n$ consider the Euler-Poincar\'e complex
of $\zono$ with coefficients in $N_i$, shifted $i$ places to the left (and truncated appropriately):
\[
\dsum_{F\in\facets_i} N_i \rightarrow
\dsum_{F\in\facets_{i+1}}N_i\rightarrow\dots
\rightarrow N_i\rightarrow0.
\]
Clearly, whenever $F \subset F'$ are faces of $\zono$, we have an inclusion $N_i^F \hookrightarrow N_i^{F'}$, and
therefore the modules $N_i^F$ define a subcomplex
\[
(EP_i): \quad \dsum_{F\in\facets_i} N^F_i \rightarrow
\dsum_{F\in\facets_{i+1}}N^F_i\rightarrow\dots
\rightarrow N^{\zono}_i\rightarrow0.
\]
The differential $\der_i$ gives a morphism of complexes
\[
\der_i: EP_i \rightarrow EP_{i-1} [-1].
\]

For $i \ge 1$ the complex $EP_i$ decomposes as a direct sum
\[
EP_i = \dsum_{y \in L_i} EP_y,
\]
where $EP_y$ is
the Euler-Poincar\'e complex of the projected zonotope $\zono_{\ge y}$ with coefficients in $N_y$,
\[
(EP_y): \quad \dsum_{F\in\facets_i^{\ge y}}N_y\rightarrow
\dsum_{F\in\facets_{i+1}^{\ge y}}N_y\rightarrow\dots
\rightarrow N_y\rightarrow0.
\]
The complexes $EP_y$ are exact, and therefore the same is true for the complexes $EP_i$, $i \ge 1$.

For $i \ge 0$ let now $D_i = \Ker \der_i\rest_{EP_i} [-1]$ be the shifted kernel of the morphism $\der_i$, i.e.
\[
(D_i): \quad \dsum_{F\in\facets_{i+1}}\Ker\der\rest_{N_i^F}\xrightarrow{\bnd_{i+2}}
\dsum_{F\in\facets_{i+2}}\Ker\der\rest_{N_i^F}\rightarrow\dots
\xrightarrow{\bnd_n}\Ker\der\rest_{N_i^\zono}\rightarrow0.
\]
By \eqref{Nexact}, we have then for every $i \ge 1$ a
short exact sequence of complexes
\begin{equation} \label{complexexact}
0\rightarrow D_i[1]\rightarrow EP_i \xrightarrow{\der_i} D_{i-1}\rightarrow0.
\end{equation}

We can now show by descending induction that for any $i=0,\dots,n-1$ the complex $D_i$
is exact. The statement is vacuous for $i=n-1$, and for the induction step we use the
short exact sequence \eqref{complexexact} to conclude the exactness of $D_{i-1}$ from the exactness
of $D_i$ and $EP_i$ for
$i = n-1, \dots, 1$.

For $i=0$ this gives the exactness of
\[
(D_0): \quad \dsum_{e\in\edges}M_e\rightarrow\dsum_{F\in\facets_2}M_F\rightarrow\dots\rightarrow M_{\zono}
\rightarrow0,
\]
i.e.~of the complex $M_{\data}^\zono [-1]$.
Together with the exactness of \eqref{exct1} we obtain the proposition.
\end{proof}

\begin{remark}
In the case $M_a=M_0$ for all $a\in\atoms(L)$ (the atomic datum with constant coefficients, cf. \S \ref{constantcoefficients}),
we can use this argument to prove the exactness
of the Euler-Poincar\'e complex $EP$ with coefficients in $M_0$
by induction on the rank
using the exactness of the Orlik-Solomon algebra (cf. \S \ref{reviewos}).
\end{remark}


\begin{remark} \label{PiecewisePolynomials}
The following special case provides a major motivation for our constructions.
Let $\data$ be given by the defining ideals of $\arr$ (cf. \S \ref{defideals}),
i.e.~$M_0=\Sym U$ and $M_H=\gen{H^\perp}{U}$ for $H\in\arr$.
Then the graphical module $\module = \module(\zono,\data)$ is the graded $\Sym U$-module (in fact algebra) of
\emph{piecewise polynomial functions} on $U^*$ with respect to $\arr$,
i.e.~the module of all (continuous) functions on $U^*$ which restrict to polynomial functions on the chambers of $\Sigma (\arr)$
(cf.~\cite{MR1013666, MR1431267}). Corollary \ref{CorDefideals} and Proposition \ref{main} together
imply that the modified Euler-Poincar\'e complex $M_{\data}^\zono$, which is a complex of graded $\Sym U$-modules, is exact if the relation complex $V$ is
\admis; in particular, this holds for restrictions of Coxeter arrangements.
More generally, by the work of Bernstein-Lunts \cite[15.7, 15.8]{MR1299527} and Brion \cite[p.~12]{MR1431267},
for any complete simplicial fan the corresponding analog of the complex $M_{\data}^\zono$ (cf.~[ibid.] for details) is exact.
Recall that restrictions of Coxeter arrangements are simplicial.
For rational fans the algebra of piecewise polynomial functions has a natural interpretation in terms of
the equivariant cohomology of toric varieties (cf.~\cite{MR1234037, MR1481477, MR1427657}).
\end{remark}




\subsection{Modules defined by graphs of fiber zonotopes} \label{SubsectionModules}
We now use the projection construction of \S \ref{truncation}
to obtain a generalization of the situation of Remark \ref{PiecewisePolynomials}, in which $\zono$ is
replaced by a fiber zonotope $\zono_P$ and the algebra $\Sym U$ of polynomial functions on $U^*$
by the symmetric algebra $\Sym V_k$ of a higher piece $V_k$ of the relation complex.

We will work in the following setting.
Let $\arr$ be a hyperplane arrangement of rank $n$ in a real vector space $U^*$
with intersection lattice $L$, and $P \subset U$ a
subspace of codimension $0 \le k \le n-1$ in general position with respect to $\arr$ (in the sense of Definition \ref{generalposition}).
Let $\pi_P: U^* \to P^*$ be the projection, and
$\arr_P$ be the hyperplane arrangement of rank $n-k$ in $P^*$ given by the hyperplanes
$\pi_P (x)$, $x \in L_{k+1}$.
Let $\zono_P$ be a zonotope dual to $\arr_P$ (of dimension $n-k$), and
let $L_P$ be the intersection lattice of $\arr_P$.
Via the map $x \mapsto \pi_P (x)$ the atoms of $L_P$
are in bijection with the elements of $L_{k+1}$. In the following we will use the notation $\facets_i$ introduced above to refer
to the set of faces of $\zono_P$ of dimension $i$, $i = -1, \dots, n-k$.

\begin{definition}
Let $R=\Sym V_k$ be the symmetric algebra of the vector space $V_k$
and $\data = (M_0, (M_a)_{a \in \atoms (L_P)})$ the atomic datum of graded $R$-modules defined by $M_0=R$ and
\begin{equation} \label{def: Ma}
M_a = \gen{\der V_{x(a)}}{V_k}
\end{equation}
for $a \in \atoms (L_P)$, where $x(a)$ is the unique element of $L_{k+1}$ with $\pi_P (x(a)) = a$ (see \eqref{def: gen} for the notation).
The graphical module $\module=\module(\zono_P,\data)$ is a graded $R$-module, and in fact a graded $R$-subalgebra of $R^{\facets_0}$.
We call it the \emph{$k$-th order relation algebra} of the arrangement $\arr$ (with respect to $P$).
It fits into the modified Euler-Poincar\'e complex $M_{\data}^{\zono_P}$, which is a complex of graded $R$-modules.
\end{definition}

\begin{remark}
There is a remarkable direct construction of a zonotope $\zono_P$ dual to $\arr_P$ from a given zonotope $\zono$ dual to the arrangement $\arr$.
Namely, we can take $\zono_P$
to be the \emph{fiber polytope} in the sense of \cite{MR1166643}
of the projection of $\zono$ under the quotient map $\pi^P: U \to U/P$. (This follows from considering $\zono$ as the projection
of a hypercube of dimension $|\arr|$, and using [ibid., Lemma 2.3, Theorem 4.1].)
Recall however that we may take $\zono_P$ to be any zonotope dual to the arrangement $\arr_P$.
Only the combinatorial structure of $\zono_P$, which is unique, matters for
our constructions.
\end{remark}

The following theorem is the second main result of our paper.
Recall that the \emph{Castel\-nuo\-vo-Mumford regularity} $\reg M$ of a finitely generated graded module $M$ over a polynomial ring $R$
is the minimal integer $r$ such that in a minimal graded free resolution
\[
\dots\rightarrow F_j\rightarrow\dots\rightarrow F_0\rightarrow M\rightarrow0
\]
of $M$ the module $F_i$ is generated in degrees $\le r+i$ for all $i \ge 0$
(cf.~\cite[\S 20.5]{MR1322960}). Note here that minimal free resolutions are unique up to isomorphism [ibid., Theorem 20.2].
By definition, a finitely generated graded $R$-module $M$ is generated in degrees $\le \reg M$.

\begin{theorem} \label{maintheorem}
Assume that the relation complex $V$ of $\arr$ is \admis. Let $P \subset U$ be as above and
$\module=\module(\zono_P,\data)$ be the $k$-th order relation algebra of $\arr$ with respect to $P$.
Then
\begin{enumerate}
\item The modified Euler-Poincar\'e complex $M_{\data}^{\zono_P}$ is exact.
\item The graded $\Sym V_k$-module $\module$ satisfies $\reg\module \le n-k$.
\end{enumerate}
In particular, $\module$ is generated as a graded $\Sym V_k$-module by its homogeneous elements of degree
$\le n-k$.
\end{theorem}

\begin{proof}
The first part follows directly from Corollary \ref{CorProj} and Proposition \ref{main}.

By \eqref{def: MF} and \eqref{def: Ma}, for any face $F$ of the zonotope $\zono_P$ the ideal $M_F$ of $R$ is generated by linear functionals.
Therefore $\reg M_F=1$ for all $F$, and the same is true for direct sums of the $M_F$.
Since the complex $M_{\data}^{\zono_P}$ is exact, we have short exact sequences
\[
0\rightarrow\Ker\bnd_i\rightarrow\dsum_{F\in\facets_{i-1}}M_F\rightarrow\Ker\bnd_{i+1}\rightarrow0,
\ \ \ i=2,\dots,n-k = \dim\zono_P.
\]
For $i=1$ we have the short exact sequence
\eqref{exct1} containing $\module$.
By the behavior of Castelnouvo-Mumford regularity in short exact sequences \cite[Corollary 20.19]{MR1322960} we obtain
\begin{gather*}
\reg(\module)\le\reg(\Ker\bnd_2),\\
\reg(\Ker\bnd_i)\le\reg(\Ker\bnd_{i+1})+1,\ \ i=2,\dots,n-k.
\end{gather*}
Since $\Ker \bnd_{n-k+1} = M_{\zono_P}$, we conclude that
$\reg(\module)\le\reg(M_{\zono_P})+n-k-1=n-k$.
\end{proof}

From Theorem \ref{roots} we infer:

\begin{corollary} \label{maincorollary}
The conclusion of Theorem \ref{maintheorem} holds if $\arr$ is a restriction of a Coxeter arrangement.
\end{corollary}

\begin{remark} \label{maintheoremremark}
The bound $n-k$ on the degrees of the generators of $\module$ is easy to prove directly from the definition
as long as $n-k \le 2$, i.e.~when the graph of $\zono_P$ has only one edge ($k=n-1$) or forms a circuit ($k=n-2$).
The bound is sharp in the case $k=0$ (the case of
Remark \ref{PiecewisePolynomials}), as shown by the precise description of the $R$-module
structure of
the algebra of piecewise polynomial functions $\module$ given in \cite[p.~12]{MR1431267} (for simplicial arrangements $\arr$). Namely, in this case $\module$ is a free
$R$-module with $h_i$ generators in degrees $i=0,\dots,n$, where
$(h_i)$ is the $h$-vector of $\zono$, i.e.,
$h_i = \sum_{j=i}^n (-1)^{j-i} {j \choose i} \card{\facets_j}$. In particular, we have $h_n = 1$.

For $k \ge 1$ the module $\module$ is in general not a free $R$-module. Explicit examples are given by
the root arrangements of type $A_{n-1}$, $n \ge 4$, and $k=n-3$ (cf.~\S \ref{higherbruhat} below).
Here $R$ is a polynomial ring in 
$n \choose 2$ many
variables. The Hilbert-Poincar\'{e} series of $\module$ can be calculated from the resolution of
Theorem \ref{maintheorem} as
\[
\sum_{k=0}^\infty \dim \module_k \ t^k = \frac{2n t + (1-t)^{n-1}}{(1-t)^{n \choose 2}},
\]
and its numerator contains at least one monomial in $t$ with a negative coefficient.
\end{remark}

\begin{remark}
Using the fact that simplicial arrangements are $2$-formal \cite{MR894288},
it is easy to see that the conclusion of Theorem \ref{maintheorem} holds also for all simplicial arrangements of rank $n\le 4$.
\end{remark}


\begin{remark} \label{Topology}
The arrangement $\arr_P$ and the zonotope $\zono_P$ appear prominently in combinatorial constructions
describing the behavior of the polyhedral fan $\Sigma (\arr)$ induced by $\arr$ under the projection $\pi_P: U^* \to P^*$ and, dually, of the
zonotope $\zono$ under the projection $\pi^P: U \to U/P$.
The arrangement $\arr_P$ governs the combinatorics of intersections
of parallel translates of $P^\perp$ with $\Sigma (\arr)$ in the following sense:
the relative interiors of the cones of $\Sigma (\arr_P)$ are precisely the fibers of the map
$\sigma_P$ from $P^*$ to subsets of $\Sigma (\arr)$ which associates
to a vector $p \in P^*$ the set of all cones $C \in \Sigma (\arr_P)$
for which the affine space $\pi_P^{-1} (p) \subset U^*$
intersects the relative interior of $C$
(\cite[Proposition 2.2]{MR1316614}, see also \cite[Ch.~9]{MR1311028}).
By choosing for every cone of $\Sigma (\arr_P)$ an arbitrary vector in its relative interior,
we can therefore view $\sigma_P$ as a map
from $\Sigma (\arr_P)$ to subsets of $\Sigma (\arr)$.

Using the terminology introduced in \cite{MR1166643}, we may pass to the dual picture and consider
the $\pi^P$-induced subdivisions of $\pi^P (\zono)$. Here, a
\emph{$\pi^P$-induced subdivision} of $\pi^P (\zono)$ is defined as a collection $\mathcal{F}$ of faces of $\zono$
such that the images $\pi^P (F)$, $F \in \mathcal{F}$, form a polyhedral complex subdividing $\pi^P (\zono)$ and
such that the relation $\pi^P (F) \subset \pi^P (F')$ for $F, F' \in \mathcal{F}$ implies
$F = F' \cap (\pi^P)^{-1} (\pi^P (F))$ \cite[Definition 9.1]{MR1311028}.
A natural partial order on the set of all $\pi^P$-induced subdivisions of $\pi^P (\zono)$ is
given by refinement of subdivisions.
The minimal elements of this poset are called \emph{tight} $\pi^P$-induced subdivisions. They can also be characterized by
the condition $\dim F = \dim \pi^P (F)$ for all $F \in \mathcal{F}$.

The face lattice of $\zono_P$ embeds canonically as a
(in general proper) subposet
of the set of all $\pi^P$-induced subdivisions \cite[Theorem 2.1]{MR1166643}. Using the identification of the face lattices of $\zono$ and $\zono_P$
with the corresponding dual polyhedral fans $\Sigma (\arr)$ and $\Sigma (\arr_P)$, this embedding is nothing else than the map $\sigma_P$ defined above.

The set of all tight $\pi^P$-induced subdivisions
carries a natural graph structure given by $\pi^P$-flips: two tight subdivisions
are connected by an edge if and only if they are the two common refinements of some (necessarily no longer tight) $\pi^P$-induced subdivision (see \cite{MR1885221},
in particular [ibid., \S 4], for a detailed exposition). Under our
general position assumption on $P$ every $\pi^P$-flip is non-degenerate in the sense of [ibid., \S 5], i.e. the corresponding
$\pi^P$-induced subdivision $\mathcal{F}$ contains a single cell $F \in \mathcal{F}$ with $\dim F = k+1$ and $\dim \pi^P (F) = k$. We label
the corresponding edge with the element $x(F) \in L_{k+1}$.
The resulting labeling of the edges of the flip graph
by the elements of $L_{k+1}$ is compatible with the natural embedding of the labeled graph of $\zono_P$ as a subgraph.
In general, the topological structure of the poset of all $\pi^P$-induced subdivisions and of the graph
of all tight $\pi^P$-induced subdivisions of $\pi^P (\zono)$ are not sufficiently well understood for our purposes --- the determination of
the homotopy type of the subdivision poset of a polytope projection (truncated by its unique maximal element)
is in fact the object of the so-called ``generalized Baues problem.''
(See \cite{MR1410159, MR1731820, MR1910043} for some general results and examples, and \cite{MR1140664, MR1205482, MR1795510} for the
case $k=1$, where the inclusion of the truncated face lattice of $\zono_P$ into the subdivision poset can be shown to be a homotopy equivalence.)
For this reason we restrict to the graph of $\zono_P$, which is topologically as simple as possible.
A series of concrete examples illuminating these concepts will be given below in \S \ref{higherbruhat}.
\end{remark}

\section{An application and a series of examples}

\subsection{Galleries and compatible families} \label{concluding}
The special case $k=1$ of \S \ref{SubsectionModules} provided the original motivation for our paper.
By Remark \ref{Topology}, we obtain here an arrangement $\arr_P$ governing the
combinatorics of straight paths through $\Sigma (\arr)$ in a fixed direction.
Moreover,
there is an interesting connection between the first-order relation algebra $\module$ and the concept of a compatible family introduced in \cite{MR2811598}.
To give some more details,
fix a chamber $\sigma_0$ of $\Sigma(\arr)$. A \emph{gallery} is a minimal
sequence $\sigma_0, \sigma_1, \dots, \sigma_N = - \sigma_0$ of adjacent chambers of $\Sigma (\arr)$, where necessarily $N = \abs{\atoms}$.
Identifying the chambers of $\Sigma (\arr)$ with the vertices of a fixed dual zonotope $\zono$,
a gallery corresponds to the sequence of vertices of
a monotone path on $\zono$ with respect to a linear functional in the chamber $-\sigma_0$.
For a subspace $P \subset U$ of codimension one, the arrangement $\arr_P$ describes the combinatorics of the intersections
of the translates of the line $P^\perp \subset U^*$ with the
cone decomposition $\Sigma (\arr)$. In concrete terms, take the fixed chamber $\sigma_0$ to be one of the two chambers intersected by $P^\perp$.
Then to every
chamber $\rho$ of the cone decomposition $\Sigma (\arr_P)$ (equivalently, to every vertex of $\zono_P$) we associate the gallery of $\arr$ consisting of the
chambers of $\Sigma (\arr)$ intersected by the affine line $\pi_P^{-1} (p)$, where $p \in \rho \subset P^*$ is arbitrary.

When two vertices $v$ and $v'$ of $\zono_P$ are connected by an edge $e$, the corresponding monotone paths on $\zono$ differ by a \emph{polygon move}
on a two-dimensional face $F$ of $\zono$ \cite{MR1731820, MR1885221}, and the associated element $x(F) \in L_2$ provides the label for the edge $e$.
If we let $\sigma_0$, \dots, $\sigma_N$ be the gallery associated to $v$, where
$\sigma_{i-1}$ and $\sigma_i$ are adjacent along the hyperplane $a_i$ of $\arr$, $i = 1, \dots, N$, then there are unique indices $1 \le i < j \le N$ such that
$\{ a_i, \dots, a_j \}$ is precisely the set of all hyperplanes of $\arr$ containing the element $x(e) \in L_2$, and the gallery associated to $v'$ is given by
\begin{equation} \label{modifiedgallery}
\sigma_0, \dots, \sigma_{i-1} = \tau_{i-1}, \tau_i, \dots, \tau_{j-1}, \sigma_j = \tau_j, \dots, \sigma_N,
\end{equation}
where the chambers $\tau_{k-1}$ and $\tau_k$ are adjacent along $a_{i+j-k}$ for $i \le k \le j$. To put it differently,
compared to the original gallery $\sigma_0$, \dots, $\sigma_N$
the hyperplanes $a_i$, \dots, $a_j$ containing $x(e)$ are
crossed in \eqref{modifiedgallery} in reverse order.

Let now $\mathcal{E}$ be an arbitrary (not necessarily commutative) finite-dimensional algebra over a field $K$ of characteristic zero.
A \emph{compatible family} $\family$ with respect to a hyperplane arrangement $\arr$ in a $K$-vector space $U^*$
is a collection of power series $\family_\sigma \in \mathcal{E} [[ U ]]$
associated to the chambers $\sigma$ of $\Sigma (\arr)$, such that
$\family_\sigma(0)=1_{\mathcal{E}}$ for all chambers $\sigma$, and
$\family_{\sigma_1\rightarrow\sigma_2}:=\family_{\sigma_1}\family_{\sigma_2}^{-1}\in
\mathcal{E}[[U_a]]$ for any two chambers $\sigma_1$, $\sigma_2$ adjacent along a hyperplane $H_a = U_a^\perp$ of $\arr$ (cf.~\cite[Definition 4.1]{MR2811598}).
Compatible families are special cases of piecewise power series with respect to $\Sigma (\arr)$.
Examples with $\mathcal{E}=K$ are given by the exponentials of piecewise linear functions on $\Sigma (\arr)$.
When $\arr$ is a restriction of a crystallographic Coxeter arrangement,
nontrivial examples for the
properly non-commutative situation
arise in the representation theory of reductive groups over local fields and adele rings (\cite[Theorem 2.1]{MR999488}, \cite[\S 2.4]{MR2811597}).

Note that for $k=1$ we have $R = \Sym V_1$ with $V_1 = \bigoplus_{a \in \atoms} U_a$. The completion of $R$
with respect to the maximal ideal generated by $V_1$ is
the power series ring $\hat{R} = K [[ V_1 ]]$.
We can view $R$ and $\hat{R}$ as a polynomial ring and a power series ring, respectively, in $N$ indeterminates which are indexed
by the set $\atoms$ of hyperplanes of $\arr$. Write $\hat{\module} = \hat{R} \module \subset \hat{R}^{\facets_0}$ for the
completion of the first-order relation algebra $\module = \module (\zono_P, \data)$.
Note that we have $\hat{\module} = \module (\zono_P, \hat{\data})$ for
$\hat{\data} = (\hat{R}, (\hat{R}\gen{\der V_{x(\alpha)}}{V_1})_{\alpha \in \atoms(L_P)})$.
For $a \in \atoms$ let
$\iota_a: \mathcal{E} [[ U_a ]] \hookrightarrow \mathcal{E} [[V_1]] = \hat{R} \otimes \mathcal{E}$ be the canonical inclusion.
Let now $\family$
be a compatible family with respect to $\Sigma(\arr)$ with values in $\mathcal{E}$.
We associate to $\family$ an element $\mu (\family) \in \hat{R}^{\facets_0} \otimes \mathcal{E}$
by setting
\[
\mu (\family) (v) = \prod_{i=1}^N \iota_{a_i} (\family_{\sigma_{i-1} \rightarrow \sigma_i}) \in \hat{R} \otimes \mathcal{E}
\]
for any $v \in \facets_0$ with corresponding gallery $\sigma_0, \dots, \sigma_N$,
where $\sigma_{i-1}$ and $\sigma_i$ are adjacent along $a_i$, $i = 1, \dots, N$.
Since the vector $\mu (\family)$ is constructed starting from a collection of power series in $\mathcal{E} [[U]]$, it is
not surprising that it cannot be an arbitrary element of $\hat{R}^{\facets_0}$. In fact, we can
easily show that it has to satisfy the following constraints.

\begin{proposition} \label{PropositionFamilies}
For any compatible family $\family$ with values in $\mathcal{E}$, the vector
$\mu (\family)$ is an element of the
submodule $\hat{\module} \otimes \mathcal{E}$ of $\hat{R}^{\facets_0} \otimes \mathcal{E}$.
\end{proposition}

\begin{proof}
Consider a pair of vertices $v, v' \in \facets_0$ adjacent along an edge $e \in \facets_1$ with $x = x(e) \in L_2$. As above, let $\sigma_0$, \dots, $\sigma_N$
and \eqref{modifiedgallery} be the galleries associated to $v$ and $v'$, respectively.
Then we have
\[
\prod_{k=i}^j \family_{\sigma_{k-1} \rightarrow \sigma_k} = \family_{\sigma_{i-1} \rightarrow \sigma_j} = \prod_{k=i}^j \family_{\tau_{k-1} \rightarrow \tau_k}
\]
as elements of the ring $\mathcal{E} [[ U_x ]]$ of power series on the two-dimensional space $U_x^\ast$. Equivalently, the difference
\[
\prod_{k=i}^j \iota_{a_k} (\family_{\sigma_{k-1} \rightarrow \sigma_k}) - \prod_{k=i}^j \iota_{a_{i+j-k}} (\family_{\tau_{k-1} \rightarrow \tau_k})
\]
lies in the kernel of the canonical evaluation map $\mathcal{E} [[ \oplus_{a \le x} U_a ]] \to \mathcal{E} [[ U_x ]]$. But this kernel is nothing else than
the two-sided ideal of $\mathcal{E} [[ \oplus_{a \le x} U_a ]]$ generated by $\der_2 V_x$. Since the remaining parts of the galleries are identical, it follows
that $\mu (\family) (v) - \mu (\family) (v')$ lies in the ideal of $\mathcal{E} [[V_1]]$ generated by $\der_2 V_x$, which shows that
$\mu (\family) \in \hat{\module} \otimes \mathcal{E}$, as asserted.
\end{proof}

\begin{remark}
Proposition \ref{PropositionFamilies} identifies certain $K$-linear constraints
satisfied by the vectors $\mu(\family)$ for compatible families $\family$.
Let $\Lambda = \Lambda (\arr)$ be the $K$-vector space of piecewise linear functions on $\Sigma (\arr)$, i.e.~the space of
all collections of vectors $\lambda_\sigma \in U$ associated
to the chambers $\sigma$ of $\Sigma (\arr)$, such that $\lambda_{\sigma_1} - \lambda_{\sigma_2} \in U_a$ whenever $\sigma_1$ and $\sigma_2$ are adjacent along a
hyperplane $H_a$ of $\arr$. Obviously, $\Lambda$ contains the vector space $U$ of linear functions on $U^\ast$.
Then for $\lambda \in \Lambda$ the collection of power series $\family_\sigma = \exp (\lambda_\sigma) \in K [[U]]$
is a compatible family with respect to $\arr$ (for $\mathcal{E}=K$), and therefore
we obtain an element $\mu (\exp (\lambda)) \in \hat{\module}$. It is easily verified that the resulting map $\Lambda \to \module_1$ in degree one induces
a linear isomorphism between the vector spaces $\Lambda / U$ and $\module_1$. The closure of the $K$-linear span of the elements $\mu (\exp (\lambda))$
is nothing else than the completion of the subalgebra of $\module$ generated by $\module_1$. It is an interesting question whether this subalgebra actually
coincides with $\module$.
(By \cite{MR1013666}, the corresponding
statement is true for the algebra of piecewise polynomial functions on a simplicial fan.)
This would mean that the relation algebra $\module$ captures in a precise sense all $K$-linear relations satisfied by the vectors $\mu (\family)$
for compatible families $\family$.
\end{remark}

The original problem motivating this paper was to prove the alternative formula of \cite[Theorem 8.1]{MR2811598} for the canonical push-forward
$\mathcal{D}_{\Sigma(\arr)} \family$ of a compatible family $\family$,
which is defined as the constant term of the power series
\[
\sum_{\sigma \in \Sigma (\arr)} \frac{\family_\sigma}{\theta_\sigma} \in \mathcal{E} [[ U ]],
\]
where $\theta_\sigma$ denotes the suitably normalized product of the linear functionals defining the walls of $\sigma$ (see [ibid., Definition 3.3] for more details).
The formula in question expresses $\mathcal{D}_{\Sigma(\arr)} \family$ by the linear terms of the basic one-variable power series
$\family_{\sigma_1\rightarrow\sigma_2} \in
\mathcal{E}[[U_a]]$ for adjacent chambers $\sigma_1$ and $\sigma_2$.
Both the definition of the
canonical push-forward and the alternative formula are easily seen to arise by base change to $\mathcal{E}$ from explicit linear functionals on
$R_n^{\facets_0}$, the space of elements of $\hat{R}^{\facets_0}$ of degree $n$:
we can write $\mathcal{D}_{\Sigma(\arr)} \family = (\mathfrak{c} \otimes \id) (\mu (\family)_n)$ for a linear functional $\mathfrak{c}$ on the $K$-vector space
$R_n^{\facets_0}$, and similarly the alternative expression of [ibid., Theorem 8.1] can be written
as $(\mathfrak{d} \otimes \id) (\mu (\family)_n)$ with a different linear functional $\mathfrak{d} \in (R_n^{\facets_0})^\ast$.
By Proposition \ref{PropositionFamilies} we have here $\mu (\family)_n \in \module_n \otimes \mathcal{E}$.
To prove the alternative formula, it is therefore enough to show that the two
linear functionals $\mathfrak{c}$ and $\mathfrak{d}$ agree on the subspace $\module_n$ of $R_n^{\facets_0}$. By Corollary \ref{maincorollary},
we know that $\module_n = V_1 \module_{n-1}$ for restrictions of Coxeter arrangements, which
allows an inductive approach to this problem.
Namely, the computation of $\mathfrak{c} (v_1 \mu_{n-1})$ and $\mathfrak{d} (v_1 \mu_{n-1})$ for $v_1 \in V_1$
and $\mu_{n-1} \in \module_{n-1}$ can be reduced (by a ``product rule'' essentially going back to Arthur in the case of the functional $\mathfrak{c}$)
to the computation of the linear functionals $\mathfrak{c}_x (\mu_{n-1,x})$
and $\mathfrak{d}_x (\mu_{n-1,x})$ corresponding to the subarrangements $\arr_{\le x}$, $x \in L_{n-1}$, which are again restrictions of Coxeter arrangements.
We will not give more details, since the alternative proof of \cite{MR2811598} is conceptually much simpler and at the same time
more general.

\subsection{The $A_n$ case, discriminantal arrangements and higher Bruhat orders} \label{higherbruhat}
As a concrete example, we finally explicate the definition of the $k$-th order relation algebra $\module$
for the root system of type $A_{n-1}$, i.e.~the braid arrangement of rank $n-1$,
and $0 \le k \le n-2$. In this case, the basic combinatorial structure of our construction is closely connected
to the discriminantal arrangements $\mathcal{B} (n,k+1)$ and the higher Bruhat orders $B(n,k+1)$ introduced by Manin-Schechtman \cite{MR1097620}.

We first describe the hyperplane arrangements and polytopes involved in our construction.
Let $U$ be an $n$-dimensional real vector space with
basis $x_1, \ldots, x_n$. The braid arrangement $\arr$ in $U^*$ is then defined by the
roots $x_i-x_j$, $1 \le i$, $j \le n$, $i \ne j$;
their span $U^L$ is the space $U_0$ of all vectors $\sum_{i=1}^n \lambda_i x_i$
with $\sum_{i=1}^n \lambda_i = 0$.
The intersection lattice $L$ of the braid arrangement is isomorphic to the lattice of
partitions of the index set $\indexset = \{1, \dots, n\}$.
A partition of $\indexset$ represents an indecomposable element of rank $k$ in $L$
if and only if all its blocks are of size one except a single block which is of size $k+1$.
Thus the set $L_{\irr,k}$ of all indecomposable elements of $L$ of a given rank $k > 0$ can and will be identified with the set
$C(\indexset,k+1)$ of all subsets of $\indexset$ of size $k+1$.
The Minkowski sum $\zono$ of the line segments $[-1,1] (x_i-x_j)$, $1 \le i < j \le n$,
is a zonotope dual to the braid arrangement $\arr$.
It is also a regular permutahedron in the space $U_0$,
namely the convex hull of the $S_n$-orbit of the vector $2 \rho = \sum_{i=1}^n (n+1-2i) x_i$.

A special feature of the root system $A_{n-1}$ is that the permutahedron $\zono$ is itself in a natural way a fiber polytope of a hypercube.
Consider the $n$-cube $C_n = \sum_{i=1}^n [-1,1] x_i$ in the space $U$ and its projection to a line segment under the map
$\pi^{U_0}: U \to U / U_0$. Then $\zono \subset U_0$ is (up to a scaling factor) just the
fiber polytope of the projection of $C_n$ by $\pi^{U_0}$ (cf.~\cite[Example 5.4]{MR1166643}, \cite[pp. 301-304]{MR1311028}).
The hyperplane arrangement dual to $C_n$ is the trivial (or Boolean) arrangement $\arr_{{\rm triv}, n}$ consisting of the $n$ hyperplanes $x_i = 0$ in $U^*$.
The space $U_0$ is in general position with respect to $\arr_{{\rm triv}, n}$, and the associated projected arrangement is the braid
arrangement $\arr$. Since the intersection lattice of $\arr_{{\rm triv}, n}$ can be identified with the Boolean algebra
$\mathfrak{P} (\indexset)$, we obtain a non-degenerate truncation map $\mathfrak{P} (\indexset)_{\ge 1} \to L$ (in fact,
this map yields an isomorphism of $L$ with the first Dilworth truncation $D_1 (\mathfrak{P} (\indexset))$, cf.~\cite[p. 302]{MR1434477}).
We recover the above identification of
$\mathfrak{P} (\indexset)_{\ge 2}$ with the set $L_{\irr}$ of all indecomposable elements of $L$.

We now consider projections $\arr_P$ of the arrangement $\arr$ (regarded as an essential arrangement in the space $U_0^*$), and
dually the fiber polytope $\zono_P$ of the projection of $\zono$ under $U_0 \to U_0 / P$
for a codimension $k$ subspace $P \subset U_0$ in general position.
In this situation, we have the sequence of projections
$U \to U /P \to U / U_0$. By \cite[Theorem 2.1]{MR1316614},
the fiber polytope $\zono'_P$ of the projection of $C_n$ via $\pi^P$ is a Minkowski summand of the zonotope $\zono_P$.
Equivalently, the dual arrangement $\arr'_P$ of $\zono'_P$ is a subarrangement of the arrangement $\arr_P$.
In fact, using [ibid., \S 2] it is easy to see that $\arr'_P$ is the
hyperplane arrangement in the space $P^*$ defined by the images
of the elements of $L_{\irr,k+1} \simeq C (\indexset,k+2)$
under $\pi_P$.
We will use the zonotope $\zono'_P$ and its dual arrangement $\arr'_P$ to describe the relation algebra $\module$.
Denote by $\facets'_0$
and $\facets'_1$ the sets of vertices (resp. edges) of $\zono'_P$.
There is obviously a canonical surjection $\facets_0 \to \facets'_0$.

More concretely, identify the elements of $U$ with their coordinate vectors with respect to the basis $(x_i)$ and
let the subspace $P \subset U_0$ be given as the nullspace of a matrix $A$ of size $k+1$ by $n$ and rank $k+1$ such that
the vector $(1, \ldots, 1)$ is a linear combination of its rows. Denote the maximal minor of $A$ associated
to $I \in C (\indexset,k+1)$ by $\alpha (I)$. Then the hyperplanes of $\arr'_P$
are defined by the vectors
\[
E_I = \sum_{\nu=0}^{k+1} (-1)^{\nu} \alpha (I \backslash \{ i_\nu \}) x_{i_\nu}, \quad
I = \{ i_0, \dots, i_{k+1} \} \in C (\indexset,k+2), \quad i_0 < \dots < i_{k+1},
\]
in the space $P$, and the fiber polytope $\zono'_P$ is (up to a scaling factor)
the Minkowski sum of the line segments $[-1,1] E_I$ \cite[Theorem 4.1]{MR1166643}.
The space $P$ is in general position with respect to the arrangement $\arr_{{\rm triv}, n}$
if and only if the minors $\alpha (I)$ for $I \in C (\indexset,k+1)$ are all non-zero (since in this case the support in $\indexset$
of $E_I$ is precisely the set $I$,
and the vectors $E_I$ are therefore pairwise non-proportional).

The labeled graph of $\zono'_P$ can be explicitly described as follows. An arbitrary subset
$\mathfrak{V}$ of $\mathfrak{P} (C (\indexset,k+2))$ has in a natural way the structure of a graph with edges labeled by the elements of $C(\indexset,k+2)$:
whenever $v \in \mathfrak{V}$, $I \notin v$, and $t' = t \cup \{ I \} \in \mathfrak{V}$, we connect $t$ and $t'$ by an edge with label $I$.
Consider the map $\nu$ from $P^* \setminus \bigcup_{H \in \arr'_P} H$ to $\mathfrak{P} (C (\indexset,k+2))$
given by
\[
\nu (\lambda) = \{ I \in C (\indexset,k+2):   \sprod{\lambda}{E_I} > 0 \}.
\]
Then the fibers of $\nu$ are precisely the interiors of the chambers of $\arr'_P$ and the set $\facets'_0$
(together with its natural structure as a labeled graph provided by $\facets'_1$)
can be identified with the image $\mathfrak{V}_P$ of $\nu$.

We now turn to the relation complex and the relation algebra. As worked out in the appendix, the relation
complex $V$ of the braid arrangement $\arr$ is given by the irrelevant ideal of the exterior algebra of the base space $U$.
For any $k \ge 0$ the space $V_k$ has a basis consisting of vectors $x_I$ indexed by the sets $I \in C(\indexset,k+1)$, and in case $k > 0$
the graded piece $V_I \subset V_k$ corresponding to such a set $I$, regarded as an element of $L_{\irr,k}$,
is the one-dimensional space spanned by $x_I$.
The ring $R = \Sym V_k$ can therefore be regarded as the polynomial ring in the variables
$x_I$, $I \in C(\indexset,k+1)$. The differential $\der: V_{k+1} \to V_k$ is given by
\[
\der x_I = \sum_{\nu=0}^{k+1} (-1)^\nu x_{I \backslash \{ i_\nu \}}, \quad
I = \{ i_0, \dots, i_{k+1} \} \in C (\indexset,k+2), \quad i_0 < \dots < i_{k+1}.
\]

Since the relation complex is supported
on indecomposables, we can view the elements of the relation algebra $\module$ as functions on the set $\facets'_0$ (the vertex set
of $\zono'_P$) satisfying congruence conditions defined by
the elements of $\facets'_1$.
In this way, $\module$ can be identified
with the algebra of all functions $m: \mathfrak{V}_P \to R$ such that
\[
m (v) \equiv m (v \cup \{ I \}) \pmod{\der x_{I}}
\]
for all
$v \in \mathfrak{V}_P$ and $I \in C (\indexset,k+2)$ with $v \cup \{ I \} \in \mathfrak{V}_P$.

The hyperplane arrangement $\arr'_P$ is a so-called discriminantal arrangement of type $(n,k+1)$ (\cite{MR1097620}, \cite[pp. 205--207]{MR1217488},
\cite{MR1209098, MR1456579, MR1720104}).
An interesting special case is obtained when the minors $\alpha (I)$ are all positive, which can be achieved by letting $A$
be a Vandermonde matrix $(t_j^{i-1})_{1 \le i \le k+1, \, 1 \le j \le n}$ for an increasing sequence
$t_1 < \dots < t_n$. In this case the sets $\mathfrak{V}_P$ defined above are subsets of the
purely combinatorial object $B(n,k+1)$ \cite{MR1097620, MR1217068}, which is the set of
all \emph{consistent} subsets of $C (\indexset,k+2)$ (in the sense of \cite[Lemma 2.4]{MR1217068}) with the partial order given by single-step inclusion.
In the setting of Remark \ref{Topology},
the set $B(n,k+1)$ can be identified with the set of all tight $\pi^P$-induced subdivisions of the projection of $C_n$ under $\pi^P$,
i.e.~the set of \emph{cubical tilings} of the zonotope $\pi^P (C_n)$,
and the natural graph structure on $B (n,k+1)$ coincides with the structure induced by $\pi^P$-flips (i.e.~cube flips, cf.~\cite{MR1731820}).
The subgraphs defined by the sets $\mathfrak{V}_P$ are precisely the
zonotopal subgraphs of $B(n,k+1)$ studied by Felsner-Ziegler \cite{MR1861424}.
For $k=0$ the set $\mathfrak{V}_P$ coincides with $B(n,1)$, which can also be described as the set of inversion sets of permutations of $\indexset$,
and therefore be identified as a poset with the symmetric group $S_n$ with its weak Bruhat order.
For $k \ge 1$ the zonotopal subgraphs $\mathfrak{V}_P$ are usually proper
subgraphs of $B(n,k+1)$, and the full graphs $B(n,k+1)$ are in general not polytopal. (For example, for $k=1$ and $n \ge 6$ every graph $\mathfrak{V}_P$
is a proper subgraph of $B(n,2)$, cf.~\cite{MR1861424}.)
Note also that the full poset of all proper $\pi^P$-induced subdivisions is known to have the homotopy type of a sphere in this situation
(\cite{MR1215321}, cf.~also \cite{MR1856417}).

\appendix
\section{Relation complexes for non-exceptional root systems}
In this appendix we explicate the relation complexes for
the infinite families $A_n$, $B_n$ and $D_n$ of indecomposable Coxeter arrangements and for their restrictions.

\subsection*{The $A_n$ case}
Let $U$ be a vector space of dimension $n+1$ with basis
$e_1,\dots,e_{n+1}$ and the standard scalar product
and consider the (non-essential) reflection
arrangement of type $A_n$ given by
the vectors $e_i-e_j$, $1\le i<j\le n+1$.
(The restrictions of the arrangements $A_n$ are also of type $A$, so they do not need
to be considered separately.)
Then the intersection lattice $L$ is the partition lattice on $\indexset = \{1,\dots,n+1\}$
with rank function $\rk(\{ I_1,\dots,I_m \})=\sum_{j=1}^m(\card{I_j}-1)$.
Let $V$ be the irrelevant ideal of the exterior algebra $\wedge^{\bullet} U$.
As a vector space, $V$ has a basis consisting of the vectors
$e_I=e_{i_1}\cdots e_{i_k}$, $I=\{i_1,\dots,i_k\}$, $1\le i_1<\dots<i_k\le n+1$,
$k\ge1$. We also set $e_\emptyset=0$.
We grade $V$ by $\deg(e_I)=\card{I}-1$. We have the differential
\[
\der e_I=\sum_{j=1}^k(-1)^{j+1}e_{I\setminus\{i_j\}}.
\]
Then $V$ is compatibly $L$-graded by
\[
V_{\{ I_1,\dots,I_m \}}=\begin{cases}U,&\text{if $\{I_1,\dots,I_m\}=\{\{1\},\dots,\{n+1\}\}$},\\
\C e_{I_j},&\text{if there is a unique $j=1,\dots,m$
such that }\card{I_j}>1,\\
0,&\text{otherwise.}\end{cases}
\]
Multiplication by $e_1$ gives a contracting $L$-homotopy for $V$
and therefore $V$ is the relation complex for $A_n$ (cf.~Remark
\ref{remuniq}).
This homotopy is not equivariant with respect to the Weyl group $W=S_{n+1}$.
The $W$-equivariant homotopy constructed in \S \ref{reflarrangements} is multiplication by $\frac1{n+1}\sum_{i=1}^{n+1}e_i$ on $V$,
corresponding to the \admiss\ section $d_0 (v) = \sum_{i=1}^{n+1}e_iv$ with Coxeter number $h^{A_n} = n+1$.
Note that as a $W$-module,
\[
V_k\simeq\wedge^{k+1}(\St\oplus\C)=\wedge^k\St\dsum\wedge^{k+1}\St,
\]
where $\St$ is the standard $n$-dimensional representation of $W$.
Recall that the representations $\wedge^k\St$, $k = 0,\dots,n$, are indecomposable.

\subsection*{The other infinite families}
The restrictions of the reflection arrangements of type $B$ and $D$
are isomorphic to the hyperplane arrangements $\Phi_{n,m}$, $m\le n$, given by the vectors
$e_i\pm e_j$, $1\le i < j\le n$, and $e_i$, $1 \le i\le m$,
in an $n$-dimensional space $E$ with basis $e_1,\dots,e_n$ and the standard scalar product.
For $m=n$ we obtain a reflection arrangement of type $B_n$, and for $m=0$ an arrangement of type $D_n$.
Here the intersection lattice $L = L^{\Phi_{n,m}}$ is the lattice whose elements consist of the following data:
\begin{enumerate}
\item a (possibly empty) subset $J$ of $\{1,\dots,n\}$, with the restriction that
if $J=\{j\}$ is a singleton, then $j\le m$,
\item an unordered partition $\{ I_1,\dots,I_k \}$ of $\{1,\dots,n\}\setminus J$,
\item for each $j=1,\dots,k$ a class of functions $\epsilon_j:I_j\rightarrow\{\pm1\}$,
where we identify $\epsilon_j$ and $-\epsilon_j$.
\end{enumerate}
The rank function is $\card{J}+\sum_{j=1}^k(\card{I_j}-1)$.
Clearly, $L^{\Phi_{n,m}}$ is a sublattice of $L^{B_n}$.

For any $I\subset\{1,\dots,n\}$ with $\card{I}>1$
and a function $\epsilon:I\rightarrow\{\pm1\}$ let $x_{(I,[\epsilon])}$ be
the element of $L$ of rank $\card{I}-1$ corresponding to $J=\emptyset$, the partition
whose only subset of size $>1$ is $I$, and the class of $\epsilon$.
Similarly, for any $J\subset\{1,\dots,n\}$ (possibly empty but with the restriction on singletons above) let
$y_J\in L_{\card{J}}$ correspond to the partition of $\{1,\dots,n\}\setminus J$ where every block is a singleton.
Let $L_{A}=\{x_{(I,[\epsilon])}\}$ and $L_{B}=\{y_J\}$. For $x \in L_A$ the corresponding subarrangement
$\Phi_{n,m,\le x}$ is of type $A_{\rk (x)}$, and for $y = y_J \in L_B$
the subarrangement
$\Phi_{n,m,\le y}$ is of type $\Phi_{\card{J}, \card{J \cap \{1,\dots,m\}}}$.

We first construct the relation complexes of the arrangements of type $B_n$ (i.e.~in the case $m=n$).
Let $W$ be the quotient of the vector space with basis
\[
e_{(I,\epsilon)},\ \ \emptyset\neq I\subset\{1,\dots,n\},\epsilon:I\rightarrow\{\pm1\},
\]
by the span of the vectors $e_{(I,\epsilon)}+(-1)^{\card{I}}e_{(I,-\epsilon)}$.
We grade $W$ by $\card{I}$ and observe that $W$ forms a chain complex with
respect to the differential
\[
\der e_{(I=\{i_1,\dots,i_k\},\epsilon)}=
\sum_{j=1}^k(-1)^{j+1}\epsilon(i_j)e_{(I\setminus\{i_j\},\epsilon|_{I\setminus\{i_j\}})}.
\]
For any $\emptyset\neq J\subset\{1,\dots,n\}$
let $W_J  \subset W_{\card{J}}$ be the span of the vectors $e_{(J,\epsilon)}$ for all $\epsilon: J \to \{\pm 1\}$.
Therefore, $\dim W_J = 2^{\abs{J}-1}$.

Let $V^{B_n}$ be the mapping cone of the identity map on $(W,\der)$.
That is, $V^{B_n}_k=W_k\oplus W_{k+1}$, $k=0,\dots,n$, with the differential
\[
V^{B_n}_k \to V^{B_n}_{k-1}, \quad
(w_k,w_{k+1})\mapsto (-\der_k w_k,w_k+\der_{k+1}w_{k+1}).
\]
$V^{B_n}$ is compatibly $L^{B_n}$-graded by
\begin{gather*}
V^{B_n}_{x_{(I,[\epsilon])}}=\{0\}\times\C e_{(I,\epsilon)}\subset V^{B_n}_{\card{I}-1},\\
V^{B_n}_{y_J}=W_J\times\{0\}
\subset V^{B_n}_{\card{J}},\ \ J\ne\emptyset,\\
V^{B_n}_0=\{0\}\times W_1,\\
V^{B_n}_x=0,\ \ \ \text{if }x\notin L^{B_n}_{A}\cup L^{B_n}_{B}.
\end{gather*}
The map $d(w_{k},w_{k+1})=(w_{k+1},0)$ is a contracting $L$-homotopy for $V^{B_n}$.
We can identify $V^{B_n}_0 \simeq W_1$ with the base space $E$. Considering
$\der_1: V^{B_n}_1 \to V^{B_n}_0$ and invoking Remark \ref{remuniq}, we see that $V^{B_n}$ is the relation complex of the arrangement $B_n$.
However, the $L$-homotopy $d$ is obviously different from the homotopy constructed in \S \ref{reflarrangements}, which we will recover below.
We have $\dim V^{B_n}_y = 2^{\rk (y) -1}$ for $y \in L^{B_n}_B \setminus \{ 0 \}$.

We now construct the relation complex of $\Phi_{n,m}$ for arbitrary $m \le n$ as a subcomplex of $V^{B_n}$.
For any $j>m$ let $(U^j,\der^j)$ be the chain complex with basis elements $f_J^j$,
$J\subset\{1,\dots,n\} \setminus \{ j \}$, of degree $\card{J}+1$, and differential
\[
\der^jf_{J=\{i_1,\dots,i_k\}}^j=\sum_{l=1}^k (-1)^{l+1} f_{J\setminus\{i_l\}}^j.
\]
For any $k\ne j$ a homotopy $d^{U^j,k}$ on $U^j$ is given by
\[
f_J^j\mapsto
\begin{cases}
(-1)^{\card{\{x \in J: x < k \}}} f_{J\cup\{k\}}^j,&k\not\in J,\\
0,&\text{otherwise.}
\end{cases}
\]
We also set $d^{U^j,j}=0$.
Let $U:=\dsum_{j > m} U^j$ with the differential $\der^U=\dsum\der^j$.
Define $d^{U,k}=\dsum_{j>m}d^{U^j,k}$ and set
\[
D^U=\begin{cases}\frac1{n-1}\sum_{k=1}^nd^{U,k},&m=0,\\
d^{U,1},&\text{otherwise.}\end{cases}
\]
Then $D^U$ is a homotopy of $U$.

The map $\phi:W\rightarrow U$ given by
\[
e_{(I,\epsilon)}\mapsto\prod_{i\in I}\epsilon(i)
\sum_{x\in I, x > m} (-1)^{\card{\{i \in I: i > x\}}} \epsilon(x)f_{I\setminus\{x\}}^x
\]
is easily seen to be a surjective map of chain complexes.
Let $K \subset W$ be its kernel. Clearly, $K = \oplus_{J \neq \emptyset} K_J$ for $K_J = K \cap W_J$.
Furthermore, we have an exact sequence
\[
0\rightarrow K\oplus W[1]\rightarrow V^{B_n}\xrightarrow{(w,w') \mapsto \phi (w)} U\rightarrow0.
\]
Note also that in degree one $K_1$ is just the span of $e_1, \dots, e_m$ inside $W_1 \simeq E$.
It follows that the complex $K\oplus W[1]$ is exact and compatibly $L^{\Phi_{n,m}}$-graded.
Therefore, the relation complex of $\Phi_{n,m}$ is given by
$V^{\Phi_{n,m}}=K\oplus W[1]$. For $y = y_J \in L^{\Phi_{n,m}}_B \setminus \{ 0 \}$ we have
$\dim V^{\Phi_{n,m}}_y = \dim K_J = 2^{\card{J}-1} - \card{J \cap \{m+1, \dots, n\}}$.

Consider for any $k = 1, \dots, n$ the linear map $d^{W,k}:W\rightarrow W$ given by
\[
d^{W,k}e_{(I,\epsilon)}=\begin{cases}
(-1)^{\card{\{x \in I: x < k \}}}
\frac12(e_{(I\cup\{k\},\epsilon\cup\{(k,1)\})}-e_{(I\cup\{k\},\epsilon\cup\{(k,-1)\})}),&
k\notin I,\\
0,&\text{otherwise.}\end{cases}
\]
It is easy to check that $\phi d^{W,k}=d^{U,k}\phi$.
In particular, $d^{W,k}$ preserves $K$.
Set
\[
D^W=\begin{cases}\frac1{n-1}\sum_{k=1}^nd^{W,k},&m=0,\\
d^{W,1},&\text{otherwise.}\end{cases}
\]
Then $\phi D^W=D^U\phi$. The map
\begin{equation} \label{phihomotopy}
(w_k,w_{k+1})\mapsto
(-D_k^Ww_k+w_{k+1}-\der_{k+2}D_{k+1}^Ww_{k+1}-D_k^W\der_{k+1} w_{k+1},D_{k+1}^Ww_{k+1})
\end{equation}
is clearly an $L$-homotopy of the complex $V^{\Phi_{n,m}}$. For $m=0$ we obtain
the $L$-homotopy of the relation complex of $D_n$ constructed in \S \ref{reflarrangements}. The Coxeter number of $D_n$ is $h^{D_n} = 2(n-1)$.

If we consider the complex $V^{B_n} = W \oplus W[1]$ instead of $V^{D_n}$ and change the map $D^W$ to $\frac1{n}\sum_{k=1}^nd^{W,k}$ in \eqref{phihomotopy},
then we obtain the $L$-homotopy of the relation complex of $B_n$ constructed in \S \ref{reflarrangements}. The Coxeter number of $B_n$ is $h^{B_n} = 2n$.

\def\cprime{$'$} \def\Dbar{\leavevmode\lower.6ex\hbox to 0pt{\hskip-.23ex
  \accent"16\hss}D} \def\cftil#1{\ifmmode\setbox7\hbox{$\accent"5E#1$}\else
  \setbox7\hbox{\accent"5E#1}\penalty 10000\relax\fi\raise 1\ht7
  \hbox{\lower1.15ex\hbox to 1\wd7{\hss\accent"7E\hss}}\penalty 10000
  \hskip-1\wd7\penalty 10000\box7}
  \def\polhk#1{\setbox0=\hbox{#1}{\ooalign{\hidewidth
  \lower1.5ex\hbox{`}\hidewidth\crcr\unhbox0}}} \def\dbar{\leavevmode\hbox to
  0pt{\hskip.2ex \accent"16\hss}d}
  \def\cfac#1{\ifmmode\setbox7\hbox{$\accent"5E#1$}\else
  \setbox7\hbox{\accent"5E#1}\penalty 10000\relax\fi\raise 1\ht7
  \hbox{\lower1.15ex\hbox to 1\wd7{\hss\accent"13\hss}}\penalty 10000
  \hskip-1\wd7\penalty 10000\box7}
  \def\ocirc#1{\ifmmode\setbox0=\hbox{$#1$}\dimen0=\ht0 \advance\dimen0
  by1pt\rlap{\hbox to\wd0{\hss\raise\dimen0
  \hbox{\hskip.2em$\scriptscriptstyle\circ$}\hss}}#1\else {\accent"17 #1}\fi}
  \def\bud{$''$} \def\cfudot#1{\ifmmode\setbox7\hbox{$\accent"5E#1$}\else
  \setbox7\hbox{\accent"5E#1}\penalty 10000\relax\fi\raise 1\ht7
  \hbox{\raise.1ex\hbox to 1\wd7{\hss.\hss}}\penalty 10000 \hskip-1\wd7\penalty
  10000\box7} \def\lfhook#1{\setbox0=\hbox{#1}{\ooalign{\hidewidth
  \lower1.5ex\hbox{'}\hidewidth\crcr\unhbox0}}}
\providecommand{\bysame}{\leavevmode\hbox to3em{\hrulefill}\thinspace}
\providecommand{\MR}{\relax\ifhmode\unskip\space\fi MR }
\providecommand{\MRhref}[2]{%
  \href{http://www.ams.org/mathscinet-getitem?mr=#1}{#2}
}
\providecommand{\href}[2]{#2}

\end{document}